\documentclass[11pt,reqno]{amsart}
\usepackage{color, amsmath,amssymb, amsfonts, amstext,amsthm, latexsym}
\UseRawInputEncoding
\allowdisplaybreaks
\newcommand{\RR}{{\mathbb R}}
\newcommand{\HH}{{\mathbb H}}

\newcommand{\EE}{{\mathbb E}}

\newcommand{\LL}{{\mathbb L}}
  
\newcommand{\PP}{{\mathbb P}}
\newcommand{\WW}{{\mathbb W}}
\newcommand{\VV}{{\mathbb V}}

\newcommand{\bu}{{\bf u}}
 \newcommand{\bU}{{\bf U}}
 \newcommand{\bQ}{{\bf Q}^0_h}
 \newcommand{\bE}{{\bf E}}
\newcommand{\bPhi}{{\bf \Phi}}

\newcommand{\Zt}{\widetilde Z}

\numberwithin{equation}{section}
\newtheorem{theorem}{Theorem}[section]
\newtheorem{defn}[theorem]{Definition}

\newtheorem{lemma}[theorem]{Lemma}
\newtheorem{remark}[theorem]{Remark}

\date{\today}

\begin{document}
\title[Strong Convergence without localization]
{Strong rates of convergence  of space-time discretization schemes   
 \\ for the  2D Navier-Stokes equations with additive noise}

\author[H. Bessaih]{Hakima Bessaih}
\address{University of Wyoming, Department of Mathematics and Statistics, Dept. 3036, 1000
East University Avenue, Laramie WY 82071, United States}
\email{ bessaih@uwyo.edu}

\author[A. Millet]{ Annie Millet}
\address{SAMM, EA 4543,
Universit\'e Paris 1 Panth\'eon Sorbonne, 90 Rue de
Tolbiac, 75634 Paris Cedex France {\it and} Laboratoire de
Probabilit\'es, Statistique et Mod\'elisation, UMR 8001, 
  Universit\'es Paris~6-Paris~7} 
\email{amillet@univ-paris1.fr}

\thanks{  Hakima Bessaih was partially supported by the Simons Foundation grant 582264. }  



\begin{abstract}
 We consider the strong solution of the 2D Navier-Stokes equations in a  torus  subject to  an additive noise. 
We implement a fully implicit time numerical scheme and a finite element method in space.  
We prove that the rate of convergence of the schemes is $\eta\in[0,1/2)$ in time  and 1 in space. Let us mention that the coefficient $\eta$ is equal to the time
 regularity of the solution with values in $\LL^2$.  Our method relies on the existence of finite  exponential moments for both the solution and its 
 time approximation. 
 Our main idea is to use a discrete Gronwall lemma for the error estimate without any  localization. 
\end{abstract}

\maketitle
 

\section{Introduction}\label{s1} \smallskip
Numerical schemes and algorithms have been introduced to best approximate and construct solutions for PDEs. 
Many algorithms  based on either finite difference, finite element  or spectral Galerkin methods 
(for the space discretization), and on either Euler schemes, Crank-Nicolson or Runge-Kutta schemes (for the time discretization) 
have been introduced for both the linear and nonlinear cases. Their rates of convergence have been widely investigated.
The literature on numerical analysis for SPDEs is now very extensive. 
In \cite{Ben1}, the models are either linear, have global Lipschitz properties,
 or more generally some monotonicity property. In this case the convergence is proven to be in mean square. 
When nonlinearities are involved that are not of Lipschitz or monotone type,  then a rate of  convergence 
 in mean square is more difficult to  obtain.  
  Indeed, because of the stochastic perturbation, one may not use the Gronwall lemma after taking the expectation of the
   error bound   
  since it involves a nonlinear term which is often  quadratic. 

In this paper we study  the  so-called incompressible Navier-Stokes equations, which describes the dynamics (velocity and pressure) of an 
incompressible fluid flow. 
These equations are parametrized by the viscosity coefficient $\nu>0$. 
Their quantitative and qualitative properties depend on the dimensional setting. 
For example, while the well posedness of global weak solutions of the 2D Navier-Stokes   is  well known and established, 
 the uniqueness of global weak solutions for the 3D case is completely open.
We will focus on the 2D incompressible Navier-Stokes equations in a bounded domain $D= [0,L]^2$,
subject to an external additive  noise  defined as:
\begin{align} \label{2D-NS}
 \partial_t u - \nu \Delta u + (u\cdot \nabla) u + \nabla \pi & = dW\quad \mbox{\rm in } \quad (0,T)\times D,\\
 \mbox{\rm div }u&=0 \quad \mbox{\rm in } \quad (0,T)\times D,
 \end{align}
 where  $T>0$. The process 
$u: \Omega\times (0,T)\times D  \to \RR^2$  is  the velocity field 
 with initial condition $u_0$ in $D$ and periodic boundary conditions $u(t,x+L v_i)=u(t,x)$ on  $(0,T)\times \partial D$, 
 where $v_i$, $i=1,2$
 denotes the canonical basis of $\RR^2$, and $\pi : \Omega\times (0,T)\times D  \to \RR$  is  the  pressure. 
  The external force is described by an additive stochastic perturbation and will be defined in detail later. 
Let $(\Omega, {\mathcal F}, ({\mathcal F}_t),  \PP)$ denote a filtered probability space and $W= (W(t), t\geq 0)$ be a Wiener process to be precisely defined later on.

Various space-time numerical schemes have been studied for the stochastic Navier-Stokes equations with a multiplicative noise, that is where in the
right hand side of  \eqref{2D-NS} we replace $dW$ by $G(u) \, dW$, where $G$ is a Lipschitz function with at most linear growth.
 We refer to   \cite{ BeBrMi, BrCaPr,Dor, Breckner, CarPro, BreDog}, where convergence in probability is  obtained  with various rates.  
 As stated previously, the main tool to get the convergence in probability  is the localization of the nonlinear term over a space of large probability.
  We studied the strong (that is  $L^2(\Omega)$) rate of convergence of the time  implicit Euler scheme (resp. space-time implicit Euler scheme coupled with
  finite element space discretization) in our previous papers \cite{Be-Mi_time}  (resp.  \cite{BeMi-FE}) for an $H^1$-valued initial condition.  
  The method is based on the fact that the solution (and the scheme) have finite  moments (bounded uniformly on the mesh). 
  For a general multiplicative noise, the rate is logarithmic. When the diffusion coefficient is bounded (which is a slight extension of an additive noise),
  the solution has exponential  moments;  we used this property in \cite{Be-Mi_time} and \cite{BeMi-FE} to  get an explicit polynomial strong rate
   of convergence. However, this rate depends on the viscosity and the strength of the noise, and is strictly less than 1/2 for the time parameter (resp. than $1$ for the
   spatial one). For a given viscosity, the rates of convergence increase to 1/2 (resp. $1$) when the noise's intensity on the time interval $[0,T]$
    approaches 0. Furthermore, the
   speed of convergence for the time discretization requires either some link between the space and time discretization parameter for general finite elements
   which satisfy the discrete LBB-condition,
   or divergence free finite elements such as Scott-Vogelius mixed elements. 
    Note that due to the stochastic forcing term, the time speed of convergence cannot exceed 1/2, which is the optimal exponent. 
    \smallskip
    
    In this paper, our aim is to obtain the best rates of convergence for the time and full (that is space-time) Euler schemes for a purely additive noise and an 
    $H^1$-valued initial condition. 
    These rates were obtained for the convergence in probability of the full scheme with a multiplicative noise and an $H^2$-valued initial
    condition in \cite{BreDog}. The proofs rely on some
    estimates already proven in \cite{BeMi-FE}, and on the existence of exponential moments for $\sup_{t\in [0,T]}  \big( |A^{1/2} u(t)|_{\LL^2}^2 
    + \int_0^t |A u(s)|_{\LL^2}^2 ds\big) $, and for a similar quantity for the time discretization (seeTheorem \ref{exp-mom}).  
    The fact that we have a purely additive perturbation
     implies that the noise disappears from the difference of the true solution and its approximation on the
    time grid. This enables us to use the Gronwall lemma for almost every $\omega$; the corresponding upper bound is $Z e^{Tg}$, where $g= C \sup_{t\in [0,T]} 
    |A^{1/2} u(t)|_{\LL^2}^2 $ and $Z$ has  finite moments of all orders. Under some conditions on  the strength of the noise in terms of  the viscosity and the
    terminal time $T$,
    we prove that the implicit time Euler scheme converges in $\LL^2$ with rate almost 1/2, and that the full space-time scheme converges with rate
    $h^\eta +k$ with $\eta \in (0,\frac{1}{2})$, where $h$ is the time mesh and $k$ is the scaling factor of the finite elements. Unlike \cite{BeMi-FE}
    we do not need some constraint between $h$ and $k$ when the finite elements are general and satisfy the discrete LBB condition. Divergence free finite elements
    only enable to have less constraints on the noise, but does not affect the speed of convergence. The proof of the rate of convergence of the implicit time scheme is
    very simple (see Section \ref{sec_Euler}).  The proof of the rate of convergence for the full scheme is based on  the same strategy, using the difference between
     the time and
    full schemes;  it    requires some more technical upper estimates on
     various error terms before using the Gronwall lemma. Some of these estimates were already proven in \cite{BeMi-FE};  we include 
     the arguments for the sake of completeness.  Note that the decomposition we use for the bilinear term is similar to that in \cite{Be-Mi_time} but
     different from that in \cite{CarPro} and provides   better estimates.  We at first prove the strong convergence for a deterministic initial condition $u_0\in V$
     under some constraints on the ``strength of the noise" described by the trace of the corresponding covariance operator $Q$.
        We also study the case of a random initial condition $u_0$ such that $\|u_0\|_V^2$ has exponential moments of order $\gamma_0$. In that case, we
     have to impose a lower bound on $\gamma_0$ in terms of $\nu $ and $T$, as well as a related upper bound on ${\rm Tr}(Q)$, and some balance between the
     integrability condition of $u_0$ and the strength of the noise has to be fulfilled. For example, if $u_0$ is a Gaussian
     $V$-valued random variable independent on the noise with covariance operator $Q_0$, this extra condition requires that the trace of $A^{\frac{1}{2}} Q_0$
      is ``small". 
   \smallskip
   
   The paper is organized as follows. In section \ref{preliminary}  we recall some results on the solution to \eqref{2D-NS}. 
   Section \ref{sec_Euler} is devoted to establish the strong speed
   of convergence of the fully implicit Euler scheme and state the existence of exponential moments of the square of its $H^1$ norm uniformly in time. 
   In section \ref{s4}  we give upper estimates of the difference between the time and space time schemes; using section \ref{sec_Euler} we deduce
   the strong speed of convergence of the full scheme for general finite elements which satisfy the discrete LBB condition, and divergence-free finite elements.  
   The above strong  convergence results are proven for  a deterministic initial condition $u_0 \in V$. 
   In section \ref{s-u0-random}, we establish similar results for a random initial condition $u_0$ with exponential moments, 
    and sketch the changes in the corresponding
   proofs.   Section \ref{s-proofs-time} provides the proofs  of time regularity results of the solution $u$ to \eqref{2D-NS} which
   together with the existence of exponential moments of $\sup_{t\in [0,T]} |A^{1/2}u(t)|_{\LL^2}^2$ is a key ingredient of the speed of convergence of the
   time scheme. In section 
   \ref{s-exp-mom-Euler} we prove the existence of exponential moments  for the time scheme, both in $V$ uniformly on the time grid, and for 
   a  ``discrete" analog of the time integral of the $H^2$-norm. This refines previous results proved in \cite{BeMi-FE} and is needed to deal with general 
   finite element discretization, when we only require the discrete LBB condition.

   \smallskip

As usual, except if specified otherwise, $C$ denotes a positive constant that may change from line to line,
 and $C(a)$ denotes  a positive constant depending on the parameter $a$.

\section{Notations and preliminary results}\label{preliminary} 
Let ${\mathbb L}^p:=L^p(D)^2$ (resp. ${\mathbb W}^{k,p}:=W^{k,p}(D)^2$)  denote the usual Lebesgue and Sobolev spaces of 
vector-valued functions
endowed with the norms $|\cdot |_{\LL^2}$ (resp. $\|\cdot \|_{{\mathbb W}^{k,p}}$). 
 In what follows, we will  consider velocity fields that have   mean zero  
 over  $[0,L]^2$.  Let $\LL^2_{per}$ denote the subset of $\LL^2$  periodic functions with mean zero over   $[0,L]^2$, and let
\begin{align*}
  H:= &\{ u\in \LL^2_{per} \; : \; {\rm div }\;  u=0 \quad \mbox {\rm weakly in }\;  D \}, \qquad
  V:=  H \cap {\mathbb W}^{1,2}
  \end{align*}
   be  separable  Hilbert spaces.  The space $H$ inherits its inner product  denoted by $(\cdot,\cdot)$ and its norm from $\LL^2$.
   The norm in $V$, inherited from ${\mathbb W}^{1,2}$, is denoted by $\| \cdot \|_V$.   Moreover,   let  $V'$ be the dual space of $V$ 
   with respect to the Gelfand triple,  
  $\langle\cdot,\cdot\rangle$ denotes the duality between $V'$ and $V$.   \\ 
  
  Let $b:V^3 \to \RR$ denote the trilinear map defined by 
  \[ b(u_1,u_2,u_3):=\int_D  \big(u_1(x)\cdot \nabla u_2(x)\big)\cdot u_3(x)\, dx, \]
  which by the incompressibility condition satisfies  $b(u_1,u_2, u_3)=-b(u_1,u_3,u_2)$  for $u_i \in V$, $i=1,2,3$. 
  There exists a continuous bilinear map $B:V\times V \mapsto
  V'$ such that
  \[ \langle B(u_1,u_2), u_3\rangle = b(u_1,u_2,u_3), \quad \mbox{\rm for all } \; u_i\in V, \; i=1,2,3.\]
  The map  $B$ satisfies the following antisymmetry relations:
  \begin{equation} \label{B}
  \langle B(u_1,u_2), u_3\rangle = - \langle B(u_1,u_3), u_2\rangle , \quad \langle B(u_1,u_2),  u_2\rangle = 0 \qquad \mbox {\rm for all } \quad u_i\in V.
  \end{equation}
 We will use some known estimates that we recall here for the sake of completeness. In dimension 2 the Gagliardo-Nirenberg inequality implies that for
  $X:=H\cap \LL^4(D)$ we have 
  \begin{equation} \label{interpol}
  \|u\|_X^2 \leq \bar{C} \; |u|_{\LL^2} \, |\nabla u|_{\LL^2} \leq \frac{\bar{C}}{2} \|u\|_V^2 
  \end{equation}
  for some positive constant $\bar{C}$.  

Furthermore, in dimension 2, the Sobolev embedding theorem implies the existence of  a positive constant $\sigma$ such that 
\begin{equation}	\label{Sobolev}
\|u\|_{\LL^\infty} \leq \sigma |Au|_{\LL^2}, \qquad \forall u\in {\rm Dom}\, A. 
\end{equation}   

Let $A=- \Delta$ with its domain  
  $\mbox{\rm Dom}(A)={\mathbb W}^{2,2}\cap H$. 
 Let $\{ \lambda_j\}_{j\geq 1}$ be eigenvalues of $A$ with $0<\lambda_1\leq \lambda_2<...$, and let $\{\zeta_j\}_{j\geq 1}$  be the corresponding
 eigenfunctions; thus $\zeta_j\in {\mathbb W}^{2,2}\cap H$, and we suppose that $\{\zeta_j\}_{j\geq 1}$ is an orthonormal basis of $H$.
  Note that $\lambda_j$ behaves like $|j|^2$, so that  $\lambda_j\to +\infty$ as $j\to +\infty$.

 We  assume that $W$ is a $Q$-Wiener process in $H$, where $Q$ is a symmetric bounded operator in $H$ with trace-class denoted by ${\rm Tr}\, (Q)$.
  In particular, we choose 
 \begin{equation}\label{def_W}
 W(t)=\sum_{j\geq 1} \sqrt{q_j}\,  \beta_j(t)\, \zeta_j,\;  t\geq 0, 
 \end{equation} 
 where $\{q_j\}_{j\geq 1}$ is a sequence of positive numbers such that $Q \zeta_j = q_j \zeta_j$, $j=1, 2, ...$  
 and $\{\beta_j(t)\}_{j\geq 1}$ are independent one-dimensional Brownian motions on   $(\Omega, {\mathcal F},$ $  ({\mathcal F}_t)_t,  \PP)$;
 then ${\rm Tr}\, (Q)=\sum_{j\geq 1} q_j$. 
  We set
  \begin{equation}\label{alpha0}
\tilde{\alpha}_0:= \frac{ \nu}{ {\rm Tr}\,(Q)}.
\end{equation} 
We will assume throughout the following stronger condition 
 \begin{equation}		\label{K0}
K_0:=  {\rm Tr}\, (A^{\frac{1}{2}} Q A^{\frac{1}{2}}) = \sum_{j\geq 1} \lambda_j\, q_j <\infty.
 \end{equation}  
Furthermore, as usual when dealing with the solution of \eqref{2D-NS} and its time discretization, we will project the equation on divergence-free fields. 
  
  For technical reasons, we will assume throughout that the initial condition $u_0$ takes values in $V$. Furthermore, $u_0$ is assumed to be either deterministic, or an ${\mathcal F}_0$-measurable
  (hence independent of $W$) such that 
  \begin{equation}\label{gamma0}
  \EE\Big[ \exp\big( \gamma_0 |A^{\frac{1}{2}} u_0|_{\LL^2}^2 \big)\Big] <\infty, \quad  {\rm for \ some}\  \gamma_0>0,\
  \end{equation}
and  only consider {\it strong solutions}  in the PDE sense. 
   Note that if $u_0\in V$ is deterministic, then $\EE\big[ \exp\big( \gamma_0 |A^{\frac{1}{2}} u_0|_{\LL^2}^2 \big)\big]<\infty$ for  any value of $\gamma_0>0$.

  Finally, note that the following identity involving the Stokes operator $A$ and the bilinear term holds (see e.g. \cite{Tem} Lemma 3.1):
  \begin{equation}   \label{A-B}
  \langle B(u,u), Au \rangle =0, \quad u\in \mbox{\rm Dom}(A). 
  \end{equation}
 
 We define a  strong solution of \eqref{2D-NS} as follows (see Definition 2.1 in \cite{CarPro}): 
 \begin{defn}
 We say that equation \eqref{2D-NS} has a strong  solution if:
 \begin{itemize}
 \item  $u $ is an adapted $V$-valued process,
 \item $\PP$ a.s. we have $u(.,.,\omega)\in C([0,T];V) \cap L^2(0,T; \mbox{\rm Dom}(A))$,
 \item  $\PP\;  \mbox{\rm a.s.}$
  \begin{align*}
  \big(u(t), \phi\big) +& \nu \int_0^t \big( A^{\frac{1}{2}} u(s), A^{\frac{1}{2}} \phi\big) ds + \int_0^t \big\langle [u(s) \cdot \nabla]u(s), \phi\big\rangle ds 
=  \big( u_0, \phi) +\big( W(t), \phi \big) 
  \end{align*}
for every $t\in [0,T]$ and every $\phi \in V$.
  \end{itemize}
 \end{defn}
 
 As usual, by projecting  \eqref{2D-NS} on divergence free fields, the pressure term is implicitly in the space $V$ and can be recovered afterwards. 
 Proposition 2.2 in \cite{BesMil} (see also \cite{BeBrMi}, Theorem 4.1) shows the following:
  \begin{theorem} \label{strong_wp}
  Assume  condition \eqref{K0}, and that $u_0$ is a $V$-valued, ${\mathcal F}_0$-measurable  random variable such that \eqref{gamma0} is satisfied.
  Then  there exists a unique    solution $u$ to equation \eqref{2D-NS}. 
  Furthermore, for every $p\in [2,+\infty)$, we have for some positive constant  $C$
  \begin{equation}   \label{bound_u}
  \EE\Big( \sup_{t\in [0,T]} \|u(t)\|_V^{p} + \int_0^T |Au(s)|_{\LL^2}^2 \big( 1+\|u(s)\|_V^{p-2}\big) ds \Big) \leq C\big[ 1+ \EE (\|u_0\|_V^p) \big].
  \end{equation}
  \end{theorem} 
   
   Furthermore, the fact that the stochastic perturbation is additive  
   implies that the solution $u$ has exponential moments.
   The first result is stated for a deterministic initial condition $u_0\in V$ and the second one for a random one with exponential moments (see
   Theorems 4.2 and 4.3 in \cite{BeMi-FE}, and \cite{HaiMat} for a similar result in a vorticity formulation).

 \begin{theorem} 		\label{th_mom_exp_u} 
  
(i)  
Let $u_0\in V$; 
   then  for $0<\alpha \leq  \tilde{\alpha}_0$, there exists a positive constant $C(\alpha)$ such that 
\begin{equation} 	\label{exp-moments-udet}
\EE\Big[ \exp\Big( \alpha  \sup_{t\in [0,T]} | A^{\frac{1}{2}} u(t)|_{\LL^2}^2   \Big) \Big]   = C(\alpha)<\infty. 
\end{equation} 
\indent (ii) Let $u_0$ be $V$-valued,  ${\mathcal F}_0$ measurable with \eqref{gamma0}. 
Then  for $0<\alpha \leq \tilde{\alpha}_0 \frac{\gamma_0 }{\gamma_0+ \tilde{\alpha_0}}$, there exists a positive constant $C(\alpha)$ such that 
\eqref{exp-moments-udet} holds.
\end{theorem}

 The following result provides moment estimates of the $\LL^2$-norm  of time increments of $u$. Note that the fact that the initial condition belongs to
 $V$ implies that the constant $C$ below does not depend on the smallest time parameter (which can be 0).  The proof is given in Section \ref{s-proofs-time}
 for the sake of completeness. 
  \begin{lemma}\label{Holder-L2}
 Assume  that \eqref{gamma0} is satisfied.
  Then  given any $ \eta \in (0,1]$ and $q\geq 1$, there exists a constant $C$  
 such that for all  $s,t\in [0,T]$ 
\begin{equation}
 \EE\big( |u(t)-u(s)|_{\LL^2}^{2q}\big) \leq C\; \big( 1+\EE(\|u_0\|_V^{4q}\big) \, |t-s|^{\eta q} .  
  \label{holder-L2}
 \end{equation} 
  \end{lemma} 

Fix an integer $N\geq 1$ and for $j=0, 1, ..., N$, set $t_j:={j} \frac{T}{N}$.
Before describing the implicit Euler scheme based on the time grid $\{t_j\}_j$, we prove  the following moment estimates
 for time increments of the solution in the $V$ and $L^4$ norms. 
 Coupled with Lemma \ref{Holder-L2}, this will be crucial to deduce the  speed of convergence
 of the time Euler scheme. 
 It slightly differs from Lemma 2.3 in \cite{CarPro}. Indeed, in order to obtain an upper estimate which does not involve
 negative powers of the lower bound of a  time interval such as $(s,t_j)$, we integrate the time increment  $u(t_j)-u(s)$ on the interval $(t_{j-1},t_j)$ of the time grid. 
 Furthermore, we do not want to  have restrictions on the moments of the corresponding sums over $j$. The proof is given in Section \ref{s-proofs-time}.
 
 \begin{lemma}\label{holder-q}
 Assume  \eqref{K0} and  \eqref{gamma0}.
 Then given   $\eta \in (0,1)$,   for every $q\in [1,+\infty)$   
  for some constant ${C}$  
\begin{align}
& \EE\Bigg( \Bigg| \sum_{j=1}^N \! \int_{t_{j-1}}^{t_j} \!\! \big[ \|u(s)-u(t_{j-1})\|_V^2 + \|u(s)-u(t_j)\|^2_V \big] ds \Bigg|^q\Bigg) \leq 
 {C}  \Big( \frac{T}{N}\Big)^{{\eta} q},  \label{holder-V_q}\\
& \EE\Bigg( \Bigg| \sum_{j=1}^N \! \int_{t_{j-1}}^{t_j} \!\! \big[ |u(s)-u(t_{j-1})|_{\LL^4}^2 + |u(s)-u(t_j)|_{\LL^4}^2 \big] ds \Bigg|^q\Bigg) \leq 
 {C} \Big( \frac{T}{N}\Big)^{{\eta} q}. \label{holder-L4_q}
 \end{align}
 \end{lemma} 
 
 \section{ Time discretization scheme}	 \label{sec_Euler}

In this section, we  study  the fully implicit time Euler scheme  of the stochastic 2D Navier-Stokes equations introduced by 
E.~Carelli and A.~Prohl in \cite{CarPro}, and recall the strong convergence proved in \cite{Be-Mi_time}. 
 Fix $N\geq 1$, let $k=\frac{T}{N}$ denote the constant time mesh,  and let  $t_l=l \frac{T}{N}$, $l=0, \cdots, N,$ denote the time grid. 
 
\subsection{The fully implicit time Euler scheme}
{\it Let $u_0$ be a $V$-valued,  ${\mathcal F}_0$-measurable  
random variable satisfying \eqref{gamma0} and  assume that \eqref{K0} holds true;  set $\bu^0=u_0$. 
Fix $N\geq 1$ and for $l=1, \cdots, N,$ find pairs 
 $\big( \bu^l, \pi^l \big) \in  V\times L^2_{per}  $ such that $\PP$  a.s. for  all  
 $\phi \in \WW^{1,2}$ and $\psi \in L^2_{per}$, 
\begin{align}    \label{full-imp1}
\big( \bu^l - \bu^{l-1} , \phi \big) +  \frac{T}{N} \Big[ \nu \big( A^{\frac{1}{2}} \bu^l , A^{\frac{1}{2}} \phi\big) &+ \big\langle (\bu^l\cdot\nabla \bu^l), 
\phi\big\rangle \Big] - \frac{T}{N} \big( \pi^l, \mbox{\rm div } \phi)
= \big(  \Delta_l W , \phi\big) , \\
\big( \mbox{\rm div } \bu^l,\psi)& = 0,  \label{fu-im-div}
\end{align}
 where $\Delta_l W = W(t_l) - W(t_{l-1})$. } \\

In this section, our aim is to recall bounds for the strong error of this Euler time scheme. Since we are looking for a $V$-valued process, 
we define the scheme for the velocity 
 projected on divergence free fields and reformulate the algorithm as follows  (see \cite[Section 3]{CarPro}).
 \begin{align}    \label{full-imp1Bis}
\big( \bu^l - \bu^{l-1} , \phi \big) +&  \frac{T}{N} \Big[ \nu \big( A^{\frac{1}{2}} \bu^l , A^{\frac{1}{2}} \phi\big) + \big\langle (\bu^l\cdot\nabla \bu^l), 
\phi\big\rangle \Big] 
= \big(   \Delta_l {W} , \phi\big) ,  \quad \forall \phi \in V.
\end{align}

The following result proves the existence and uniqueness of the solution $\{ \bu^l, \, l=0, \cdots, N\}$ of \eqref{full-imp1Bis}; 
it provides moment estimates for this solution.
Note that since $\|u_0\|_V^2$ has a finite exponential moment, all the dyadic moments of this random variable are finite.
Therefore,  we may use  the
 induction argument which relates two consecutive dyadic
numbers (see step 4 of the proof of   \cite[Lemma 3.1]{BrCaPr}) and deduce estimates for any exposant.  
\begin{lemma} \cite[Lemma 3.1]{CarPro} 		\label{moments_uN}
Assume \eqref{K0} and \eqref{gamma0}
 Then there exists a unique solution $\{\bu^l\}_{l=0}^N$ to \eqref{full-imp1Bis} with $\bu^0=u_0$. The random variables
$\bu^l$ are ${\mathcal F}_{t_l}$-measurable and belong to $L^2(\Omega;V)$ a.s. Furthermore, given any $q\in [2,\infty)$ 
\begin{align}
 \EE\Big( \max_{0\leq l\leq N} \|\bu^l\|_V^{q} + 2\nu \frac{T}{N} \sum_{l=1}^N \| \bu^l \|_V^{q-2}\; |A \bu^l |^2_{\LL^2}\Big)
&\leq C_1(T,q), \label{C(T,q)}\\
\EE\Big( \sum_{l=1}^N \|\bu^l - \bu^{l-1}\|_V^2 \; \| \bu^l\|_V^2\Big) & \leq C_2(T,2),    \label{C(T,2)}\\
\EE \Big[ \Big( \sum_{l=1}^N \| \bu^l- \bu^{l-1}\|_V^2 \Big)^{q} + \Big( \nu \frac{T}{N} \sum_{l=1}^N |A \bu^l|_{\LL^2}^2 \Big)^{q} \Big] & 
\leq C_3(T,q), 
\label{C(T,4)}
\end{align}
where  for $i=1,2,3$, $C_i(T,q)$  is a constant which does not depend on $N$.
\end{lemma}

 The Euler scheme has exponential moments similar to those of the solution $u$ to \eqref{2D-NS}; the next result partly extends
Theorems~8.1 and 8.3 in \cite{BeMi-FE}, including exponential moments of $k \nu \sum_l |A\bu^l |_{\LL^2}^2$. This refinement will be crucial to
deal with space-time approximation of the solution to \eqref{2D-NS}. The proof is given in Section \ref{s-exp-mom-Euler}.

\begin{theorem}		\label{exp-mom}
 (i) Let $u_0\in V$ be deterministic. Then for $0<\alpha < \tilde{\alpha}_0$,
 there exists a positive constant $C_1(\alpha)$ such that for $N$ large enough, 
 \begin{equation} 	\label{exp-moments-ul_al}
\EE\Big[ \exp\Big( \alpha  \max_{0\leq l\leq N}   |A^{\frac{1}{2}}  \bu^l |_{\LL^2}^2 \Big)  \Big] 
= C_1(\alpha) < \infty. 
\end{equation} 
 and for $\beta < \frac{\tilde{\alpha}_0}{2}$, there exists a positive constant $C_2(\beta)$ such that for $N$ large enough, 
\begin{equation} 	\label{exp-moments-Aul_al}
\EE\Big[ \exp\Big( \beta  \max_{0\leq n\leq N}  \Big[  |A^{\frac{1}{2}}  \bu^n|_{\LL^2}^2  + \beta  \nu \frac{T}{N} \sum_{l=1}^n |A \bu^l|_{\LL^2}^2 \Big] \Big)  \Big] 
= C_2(\beta) < \infty. 
\end{equation} 

 (ii)  Let $u_0$ be random, ${\mathcal F}_0$-measurable  and satisfy  \eqref{gamma0}.\\
Set $ \tilde{\beta}_0:= \tilde{\alpha}_0 \;  \frac{\gamma_0 }{\gamma_0+ \tilde{\alpha_0}}$. Then for $0<\alpha <\tilde{ \beta}_0$ there exists a positive constant
$C_1(\alpha)$ such that \eqref{exp-moments-ul_al} holds for $N$ large enough.\\
Set  $\tilde{\beta}_1:=  \tilde{\alpha}_0\, \frac{ \gamma_0}{2 \gamma_0+\tilde{\alpha}_0} $; then  for  $0<\beta<\tilde{\beta}_1$ there exists a constant
$C_2(\beta)$ such that \eqref{exp-moments-Aul_al} holds for $N$ large enough. 
\end{theorem} 

\subsection{Strong convergence for the time Euler scheme}

For $k=0, \cdots, N$, let $e_k:=u(t_k) - \bu^k$ denote the error of this scheme (note that $e_0=0$). 
Then for any $\phi \in V$ and $j=1, \cdots, N$,  we have
\begin{align}											\label{e-e2}
(e_{j} - e_{j-1}, \phi) & +  \int_{t_{j-1}}^{t_j}  \Big[ \nu \big( A^{\frac{1}{2}} u(s) - A^{\frac{1}{2}} \bu^j\, ,\,  A^{\frac{1}{2}} \phi\big) 
+ \big\langle  B(u(s),u(s)) - B(\bu^j,\bu^j)) \, ,\,  \phi\big\rangle \Big] ds =0. 
\end{align}
 The following strong convergence is the main result of this section. Unlike Theorem \cite[Th 4.6]{Be-Mi_time}, 
the maximal speed of convergence  is given by the coefficient $\eta$ which is the time regularity of the solution $u$. 
However, for the convergence to hold there is a constraint on ${\rm Tr} (Q)$ which reflects the strength of the noise. 
\begin{theorem} \label{Loc_cv_Euler} 
Let $u_0\in V$ be deterministic;  assume    condition \eqref{K0} is satisfied  and 
   ${\rm Tr}\, (Q) < \frac{2 \nu^2}{{\bar C}^2\, T} $. 
Then, given any  $\eta \in (0, 1)$ there exists a positive constant $C$ such that \begin{equation}				\label{moments1}
\EE\Big( \max_{1\leq j \leq N}    |e_j|_{\LL^2}^2   + {\nu} \frac{T}{N}
\sum_{j=1}^N   |A^{\frac{1}{2}} e_j|_{\LL^2}^2   \Big) \leq C \Big(\frac{T}{N}\Big)^\eta , 
\end{equation}
for   $N$ large enough. 
\end{theorem} 
\begin{proof} The proof contains several  steps: some upper estimates of the bilinear term,  a pointwise discrete Gronwall lemma
and moment estimates. We rewrite \eqref{e-e2} with $\phi= e_j$.

\noindent {\bf Step 1:  Upper estimates for the bilinear term}\\
 Let us consider the duality between the difference of the bilinear terms and $e_{j}$, that is the  upper estimate of  $\int_{t_{j-1}}^{t_j} \langle
 B(u(s),u(s)) - B(\bu^j,\bu^j) \, , \, e_{j}\big\rangle ds$.  For every $s\in (t_{j-1}, t_j]$, using the bilinearity of $B$ 
and the antisymmetry property  \eqref{B}, 
we deduce 
\begin{equation} \label{dif_B-B}
  \big\langle B\big( u(s), u(s) \big) -  B\big( \bu^j, \bu^j\big)  \, , \, e_j \big\rangle = \sum_{i=1}^3 T_i(s),
  \end{equation}
where, since  $\big\langle  B\big( v,  \bu^j \big)  \, , \, e_j \big\rangle = \big\langle  B\big( v,  u(t_j)  \big)  \, , \, e_j \big\rangle$ for every $v\in V$,
\begin{align*}
T_1(s):= & \big\langle  B\big( e_j ,  \bu^j \big)  \, , \, e_j \big\rangle = \big\langle B\big(e_j , u(t_j)\big)  \, , \, e_j \big\rangle,\\ 
T_2(s):= &\big\langle  B\big(  u(s) -u(t_j) , u(t_j) \big) \, , \, e_j \big\rangle, \\
T_3(s):= & \big\langle B\big(u(s) , u(s)-u(t_j)\big) \, , \, e_j \big\rangle = - \big\langle B\big(u(s) , e_j\big) \, , \, u(s)-u(t_j) \big\rangle.
\end{align*}
Note that, unlike the first formulation of $T_1(s)$, the second  one only depends on the error and on the solution to \eqref{2D-NS}, 
and not on the approximation scheme. The H\"older  inequality and \eqref{interpol} yield  
 for every $\delta_1 >0$ and $\bar{C}$ defined in the interpolation
inequality \eqref{interpol} 
\begin{align} 		\label{maj_T1}			 
 \int_{t_{j-1}}^{t_j} | T_1(s) |\, ds &\leq \bar{C}  \frac{T}{N} \,  |  e_j|_{\LL^2}  \, |A^{\frac{1}{2}} e_j|_{\LL^2} \,   | A^{\frac{1}{2}} u(t_j)|_{\LL^2}   \nonumber\\
& \leq \delta_1 \nu \frac{T}{N} \, |A^{\frac{1}{2}} e_j|_{\LL^2}^2  + \frac{\bar{C}^2}{4\delta_1 \nu} \, \frac{T}{N} \,  |e_j|_{\LL^2}^2 \, |A^{\frac{1}{2}} u(t_j)|_{\LL^2}^2, 
\end{align}
where the last upper estimate follows from the Young inequality (with conjugate exponents 2 and 2).
A similar argument using the H\"older and Young inequalities with exponents  4, 4 and 2 implies that for any $\delta_2 >0$ and $\gamma_2>0$, 
\begin{align*}	
|T_2(s)| \leq \delta_2 \nu |A^{\frac{1}{2}} e_j|_{\LL^2}^2 + \gamma_2 |e_j|^2_{\LL^2} + 
 C(\nu, \delta_2, \gamma_2)   \|u(t_j)-u(s)\|_{\LL^4}^2 |A^{\frac{1}{2}} u(t_j)|_{\LL^2}^2 . 
\end{align*}
Using the Cauchy-Schwarz inequality   we deduce 
\begin{align}			\label{maj_T2}
 \int_{t_{j-1}}^{t_j} \!\! | T_2(s) |\, ds \leq  \;  &\delta_2 \nu  \frac{T}{N}  |A^{\frac{1}{2}} e_j|_{\LL^2}^2  
 + \gamma_2 \frac{T}{N}    |e_j|_{\LL^2}^2  \nonumber \\
& \; +  
 C(\nu, \delta_2, \gamma_2)   |A^{\frac{1}{2}} u(t_j)|_{\LL^2}^2 \int_{t_{j-1}}^{t_j}\!\!  \|u(t_j)-u(s)\|_{\LL^4}^2 ds . 
\end{align} 
Similar computations using  the H\"older and Young  inequalities imply 
\begin{align}				\label{maj_T3}
\int_{t_{j-1}}^{t_j} | T_3(s) |\, ds \leq & \; \delta_3 \nu \frac{T}{N} |A^{\frac{1}{2}} e_j|_{\LL^2}^2  
+ \frac{1}{4\nu \delta_3}  \int_{t_{j-1}}^{t_j} \|u(s)\|_{\LL^4}^2 \, \|u(t_j)-u(s)\|_{\LL^4}^2 ds
\end{align}
for any $\delta_3>0$. 

Furthermore, note that 
\[ 
 \nu \int_{t_{j-1}}^{t_j} \!\! \big( A^{\frac{1}{2}}(u(s) - \bu ^j)\, , \, A^{\frac{1}{2}} e_j \big)  ds =  \nu  \frac{T}{N}  |A^{\frac{1}{2}} e_j|_{\LL^2}^2 + \nu  \int_{t_{j-1}}^{t_j} \!\!
\big( A^{\frac{1}{2}}(u(s)-u(t_j)) ,  A^{\frac{1}{2}} e_j\big) ds. 
\]
Using the Cauchy-Schwarz and Young inequalities, we deduce 
\[ \nu \int_{t_{j-1}}^{t_j} \!\!
\big| \big( A^{\frac{1}{2}}(u(s)-u(t_j)) ,  A^{\frac{1}{2}} e_j\big)\big|  ds  \leq 
\delta_0 \nu \frac{T}{N} |A^{\frac{1}{2}} e_j|_{\LL^2}^2 + \frac{\nu }{4\delta_0 } \int_{t_{j-1}}^{t_j} \big|A^{\frac{1}{2}}  \big(u(s) - u(t_j) \big)\big|_{\LL^2}^2 ds 
\] 
for any $\delta_0>0$.  Hence using the above upper estimates \eqref{maj_T1} -- \eqref{maj_T3} in \eqref{e-e2} with $\phi = e_j$, we deduce
\begin{align}						\label{(e-e,e)}
\big( e_j-e_{j-1}\, , e_j\big) + \nu \frac{T}{N} |A^{\frac{1}{2}} e_j|_{\LL^2}^2 \leq  & \; \nu \sum_{r=0}^3 \delta_r \frac{T}{N} |A^{\frac{1}{2}} e_j|_{\LL^2}^2 
+ \sum_{l=1}^3 \tilde{T_j}(l)\nonumber\\
&\quad +  \Big( \gamma_2+\frac{\bar{C}^2}{4\delta_1 \nu} |A^{\frac{1}{2}} u(t_j)|_{\LL^2}^2 \Big) \, \frac{T}{N} \, |e_j|_{\LL^2}^2, 
\end{align}
where 
\begin{align*}
\tilde{T_j}(1) =  & \; \frac{\nu }{4 \delta_0 } \int_{t_{j-1}}^{t_j} \big|A^{\frac{1}{2}} \big(u(s) - u(t_j) \big)\big|_{\LL^2}^2 ds, \\
\tilde{T_j}(2) =  & \;   C(\nu, \delta_2, \gamma_2)   \, |A^{\frac{1}{2}} u(t_j)|_{\LL^2}^2 \int_{t_{j-1}}^{t_j} \|u(s)-u(t_j)\|_{\LL^4}^2 ds, \\
\tilde{T_j}(3) =  & \;  \frac{1}{4\nu \delta_3} \int_{t_{j-1}}^{t_j} \|u(s)\|_{\LL^4}^2  \|u(t_j)-u(s)\|_{\LL^4}^2 ds. 
\end{align*}

\noindent {\bf Step 2:  Pointwise discrete Gronwall lemma}\\
We will use the convention $\sum_{j=l_1}^{l_2} x_j =0$ if $l_1>l_2$. 
Adding the inequalities \eqref{(e-e,e)}, 
 using  $e_0=0$ and the identity 
$(a,a-b) = \frac{1}{2} \big[ |a|_{\LL^2}^2 - |b|_{\LL^2}^2 + |a-b|_{\LL^2}^2\big]$, we deduce for $k=1, \cdots, N$
\begin{align}					\label{maj_max} 
& \max_{1\leq j\leq k} \Big[  \frac{1}{2} \Big( |e_j|_{\LL^2}^2 +  \sum_{l=1}^j |e_l-e_{l-1}|_{\LL^2}^2 \Big)
  +   \nu \Big(1-\sum_{r=0}^3 \delta_r \Big)\, \frac{T}{N} \sum_{l=1}^j  |A^{\frac{1}{2}} e_l |_{\LL^2}^2  \Big] \nonumber \\
&\quad  \leq   \frac{T}{N} \, \sum_{j=1}^k\Big( \gamma_2+\frac{\bar{C}^2}{4\delta_1 \nu} |A^{\frac{1}{2}} u(t_j)|_{\LL^2}^2 \Big) \,  |e_j|_{\LL^2}^2    
+  \sum_{i=1}^3 \sum_{j=1}^k \tilde{T_j}(i)\nonumber\\
&\quad \leq 2 \gamma_2\frac{T}{N}\Big( |u(t_k)|_{\LL^2}^2 +|\bu^k|_{\LL^2}^2 \Big) 
+ \frac{\bar{C}^2}{2\delta_1 \nu}\frac{T}{N} |A^{\frac{1}{2}} u(t_k)_{\LL^2}^2 \Big( |u(t_k)|_{\LL^2}^2 +|\bu^k|_{\LL^2}^2 \Big)  \nonumber\\
&\quad \quad + \frac{T}{N} \, \sum_{j=1}^{k-1}\Big( \gamma_2+\frac{\bar{C}^2}{4\delta_1 \nu} |A^{\frac{1}{2}} u(t_j)|_{\LL^2}^2 \Big) \,  |e_j|_{\LL^2}^2    
+  \sum_{i=1}^3 \sum_{j=1}^k \tilde{T_j}(i)\nonumber\\
&\quad \leq  2 \gamma_2\frac{T}{N}\Big(\sup_{s\in [0,T]} |u(s)|_{\LL^2}^2 +\sup_{1\leq k\leq N} |\bu^k|_{\LL^2}^2 \Big) 
+ \frac{\bar{C}^2}{2\delta_1 \nu}\frac{T}{N} \Big( \frac{3}{2} \sup_{s\in [0,T]} \|u(s)\|_{V}^4 +\frac{1}{2}\sup_{1\leq k\leq N} |\bu^k|_{\LL^2}^4 \Big)  \nonumber\\
&\quad \quad + \frac{T}{N} \, \sum_{j=1}^{k-1}\Big( \gamma_2+\frac{\bar{C}^2}{4\delta_1 \nu} \sup_{s\in [0,T]} |A^{\frac{1}{2}} u(s)|_{\LL^2}^2 \Big) \,  |e_j|_{\LL^2}^2    
+  \sum_{i=1}^3 \sum_{j=1}^N \tilde{T_j}(i)
\end{align}
\smallskip

Set 
\begin{align} 	\label{g-Z}
g:= &\;  \Big(2 \gamma_2+\frac{\bar{C}^2}{2\delta_1 \nu} \sup_{s\in [0,T]} |A^{\frac{1}{2}} u(s)|_{\LL^2}^2 \Big), \nonumber \\
Z_1:= &\; 4 \gamma_2\frac{T}{N}\Big(\sup_{s\in [0,T]} |u(s)|_{\LL^2}^2 +\sup_{1\leq k\leq N} |\bu^k|_{\LL^2}^2 \Big) 
+ \frac{\bar{C}^2}{2\delta_1 \nu}\frac{T}{N} \Big( 3\sup_{s\in [0,T]} \|u(s)\|_{V}^4 +\sup_{1\leq k\leq N} |\bu^k|_{\LL^2}^4 \Big) ,  \nonumber \\
Z_2:=& \;  2 \sum_{i=1}^3 \sum_{j=1}^N \tilde{T_j}(i).
\end{align}
Suppose that $\sum_{r=0}^3 \delta_r<1$; then neglecting some non-negative terms in the left hand side of \eqref{maj_max} we deduce 
for $Z:=Z_1+Z_2$ 
$$\max_{1\leq j\leq k}|e_j|_{\LL^2}^2 \leq Z +  \frac{T}{N}\, g\,  \sum_{j=1}^{k-1}  |e_j|_{\LL^2}^2, \quad k=2, ..., N.$$ 
Using the discrete Gronwall lemma,   see \cite{Holte}, we deduce that  a.s. 
\begin{equation}\label{Gronwall1}
\max_{1\leq j\leq N}|e_j|_{\LL^2}^2 
\leq Z \Big( 1+\frac{Tg}{N}\Big)^{N-1} \leq Z  e^{Tg}.
\end{equation}
\newpage
\noindent {\bf Step 3: Strong speed of convergence}   

Taking the expected value in \eqref{Gronwall1} and using the H\"older inequality with conjugate exponents $p,q\in (1,+\infty)$, we obtain
 \begin{align}		\label{maj_E_sup_ej} 
 \EE\Big( \max_{1\leq j\leq k}  |e_j|_{\LL^2}^2 \Big) &\leq  \EE\big( Z e^{Tg}\Big) \leq  
  \Big[\EE \big(  e^{pTg}\big) \Big]^{1/p}  \Big[\EE \big( Z^q\big)\Big]^{1/q} \nonumber\\
  &=e^{2\gamma_2T}\, \Big[\EE \Big( e^{pT  \frac{\bar{C}^2}{2\delta_1 \nu} \sup_{s\in [0,T]} |A^{\frac{1}{2}} u(s)|_{\LL^2}^2 }\Big) \Big]^{1/p}  
  \Big[\EE \big( |Z_1+Z_2|^q\big)\Big]^{1/q} .
      \end{align}
 Since by assumption  ${\rm Tr}\, (Q) < \frac{2\nu^2}{\bar{C}^2 \,  T}$, we have $\frac{T\, \bar{C}^2}{2\, \nu} < \frac{\nu}{{\rm Tr}±, (Q)}:=\tilde{\alpha}_0$. Therefore, 
 we may choose $p\in (1,+\infty)$ (close to 1)
and $\delta_1\in (0,1)$ (close to 1) such that $pT\frac{ \bar{C}^2}{2\delta_1 \nu} <\tilde{\alpha}_0$. 
Using \eqref{exp-moments-udet} we deduce 
\begin{equation}		\label{upper_p_g}
\EE \Big[ \exp\Big( pT  \frac{\bar{C}^2}{2\delta_1 \nu} \sup_{s\in [0,T] } |A^{\frac{1}{2}} u(s)|_{\LL^2}^2  \Big)\Big] =C <\infty. 
 \end{equation}
Choosing $p$ close enough to 1 we may have $q=2^{\tilde{q}}$ for some $\tilde{q} \in [2,+\infty)$. 
Since $\|u_0\|_V^2$ has exponential moments, we have $\EE(\|u_0\|_V^r)<\infty$ for $r$ arbitrary large, that is $r\geq 2 q$ . 
Using the estimates \eqref{bound_u} and  \eqref{C(T,q)}, we infer 
\begin{equation}		\label{maj_Z1_q}
  \big[ \EE\big( Z_1^q\big) \big]^{\frac{1}{q}} \leq C\, \frac{T}{N}.
  \end{equation} 
The time regularity in $V$ proved in \eqref{holder-V_q}  implies 
\begin{equation}	\label{maj_tildeT1_q}
\Big[ \EE\Big(\Big|\sum_{j=1}^{N}\tilde{T_j}(1)\Big|^q\Big)\Big]^{1/q}\leq C \Big(\frac{T}{N}\Big)^{\eta}, \quad  \eta\in (0, 1).
\end{equation}
On the other hand, using the Cauchy-Schwarz inequality we obtain 
\begin{align}		\label{maj_tildeT2_q}
\Big[ \EE\Big(\Big| \sum_{j=1}^{N} \tilde{T_j}(2)&\Big|^q\Big)\Big]^{1/q} =  C(\nu, \delta_2, \gamma_2)  
\Big[ \EE\Big(\Big|\sum_{j=1}^{N} \, |A^{\frac{1}{2}} u(t_j)|_{\LL^2}^2 \int_{t_{j-1}}^{t_j} \|u(s)-u(t_j)\|_{\LL^4}^2 ds \Big|^q\Big) \Big]^{1/q} \nonumber \\
&\leq C  \Big[ \EE\Big(\Big| \sup_{s\in [0,T] } \| u(s)\|_V^{2} \sum_{j=1}^{N}  \int_{t_{j-1}}^{t_j} \|u(s)-u(t_j)\|_{\LL^4}^2 ds \Big|^q\Big)\Big]^{1/q}\nonumber \\
&\leq C \Big[\EE\Big( \sup_{s\in [0,T] } \| u(s)\|_V^{4q}\Big)\Big]^{\frac{1}{2q}} 
\Big[\EE\Big(\Big| \sum_{j=1}^{N}  \int_{t_{j-1}}^{t_j} \|u(s)-u(t_j)\|_{\LL^4}^2 ds \Big|^{2q}\Big)\Big]^{\frac{1}{2q}}\nonumber\\
&\leq C  \Big(\frac{T}{N}\Big)^\eta,		\quad \eta\in( 0, 1), 
\end{align}
where the last upper estimate is a consequence of \eqref{bound_u} and \eqref{holder-L4_q}.  

Using the Gagliardo-Nirenberg inequality \eqref{interpol}, a similar argument implies 
\begin{align} \label{maj_tildeT3_q}
\Big[ \EE\Big(\Big| \sum_{j=1}^{N} \tilde{T_j}(3)\Big|^q\Big)\Big]^{1/q} &=   \frac{1}{4 \nu \delta_3} \Big[ \EE\Big(\sum_{j=1}^{N} 
\int_{t_{j-1}}^{t_j}\|u(s)\|_{\LL^4}^2 \|u(s)-u(t_j)\|_{\LL^4}^2 ds \Big)^q\Big]^{1/q} \nonumber \\
&\leq C  \Big[ \EE\Big(\Big| \sup_{s\in [0,T] } \|u(s)\|_V^{2} \sum_{j=1}^{N}  \int_{t_{j-1}}^{t_j} \|u(s)-u(t_j)\|_{\LL^4}^2 ds \Big|^q\Big) \Big]^{1/q}\nonumber \\
&\leq C  \Big(\frac{T}{N}\Big)^\eta, \quad \eta \in ( 0,1)  . 
\end{align}
The upper estimates \eqref{maj_tildeT1_q}--\eqref{maj_tildeT3_q} imply 
\begin{equation} 		\label{maj_Z2_q}
\big[ \EE\big( |Z_2|^q\big) \Big]^{1/q} \leq C\, \Big( \frac{T}{N}\Big)^\eta, \quad \eta \in (0, 1) . 
\end{equation}
 We then choose $\delta_r >0$, $r=0,2,3$ such that $\sum_{r=0}^3 \delta_r <1$. 
Collecting  the estimates \eqref{maj_E_sup_ej}, \eqref{upper_p_g}, \eqref{maj_Z1_q} and \eqref{maj_Z2_q},   we deduce  
$$\EE\Big( \max_{1\leq j\leq N}  |e_j|_{\LL^2}^2 \Big) \leq C \Big(\frac{T}{N}\Big)^{\eta}, \quad \eta \in (0,1).$$
Finally, plugging this inequality  in \eqref{maj_max}, we obtain \eqref{moments1}; this completes   the proof. 
 \end{proof}

\section{Space-time discretization}		\label{s4}
\subsection{Description of the finite element method}
When studying a space time discretization using finite elements,
 one needs to have a stable pairing of the velocity and the pressure which
 satisfy the discrete LBB-condition (see e.g. \cite{CarPro}, page 2469 and pages 2487-2489). Stability issues are crucial, and the pressure has
 to be discretized together with the velocity.
 
 Let ${\mathcal T}_h$ be a quasi-uniform triangulation of the domain $D\subset \RR^2$, using triangles of maximal diameter $h>0$, and set
 $\bar{D}=\cup_{K\in {\mathcal T}_h} \bar{K}$. Let $\PP_i(K):=[P_i(K)]^2$ denote the space of polynomial vector fields  on $K$
  of degree less than or equal to $i$. Given non negative integers $i,j$, we introduce finite element function spaces 
  \begin{align*}
  {\HH}_h &\,  := \{ {\bf U} \in C^0(\bar{D}) \cap \WW^{1,2}_{per}(D) : {\bf U}\in  \PP_i(K), \quad \forall K\in {\mathcal T}_h\}, \\
  L_h &\, := \{ \Pi \in L^2_{per}(D) : \Pi\in P_j(K), \quad \forall K\in {\mathcal T}_h\},
  \end{align*}
which satisfy the discrete LBB-condition 
\begin{equation}		\label{LBB}
\sup_{\bPhi \in \HH_h} \frac{(\mbox{\rm div }  \bPhi , \Pi)}{|A^{\frac{1}{2}} \bPhi |_{\LL^2} } \geq C\; |\Pi|_{\LL^2}, \quad \forall \Pi \in L_h,
\end{equation}
with a constant $C>0$ independent of the mesh size $h>0$.  Here $C^0(\bar{D})$ denotes the set of continuous 2-dimensional vector fields on $\bar{D}$. 

Define the subset $\VV_h \subset \HH_h$  of discrete divergence-free vector fields 
\begin{equation} 		\label{Vh}
\VV_h:= \{ \bPhi \in \HH_h \; : \; ({\rm div }\,  \bPhi, \Lambda)=0, \quad \forall \Lambda \in L_h\}.
\end{equation}
 Note that in general $\VV_h \not\subset V$. 
A way around this problem is to choose a space approximation such  that $\VV_h\subset V$, such as the Scott-Vogelius mixed elements 
(see \cite{ScoVog} and \cite{Zhang}). 
This particular case yields a better approximation. Indeed, on one hand the pressure will not appear in the upper estimate, and on the
other hand a different localization will provide a polynomial error in the case of an additive noise. 

Let ${\bf Q}_h^0 : \LL^2 \mapsto \VV_h$ (resp. $P^0_h:L^2_{per}(D)\mapsto L_h$) denote the orthogonal projection defined by
\begin{equation} 		\label{Q0h}
({\bf z}- {\bf Q}^0_h {\bf z} , {\bf \Phi})=0, \quad \forall {\bf \Phi} \in \VV_h, \qquad {\rm (resp. } \quad (z-P^0_hz, \Lambda)=0,
 \quad \forall \Lambda \in L_h{\rm )}.
\end{equation} 
The following estimates are standard (see e.g. \cite{HeyRan})
\begin{align}
|{\bf z} - {\bf Q}_h^0 {\bf z}|_{\LL^2} + h |A^{\frac{1}{2}} ({\bf z} - {\bf Q}^0_h {\bf z})|_{\LL^2} & \leq \; C\,  h^2 \,  |A{\bf z}|_{\LL^2}, \quad 
\forall {\bf z}\in V \cap
\WW^{2,2}(D),	\label{I-Q_1}\\
|{\bf z} - {\bf Q}^0_h {\bf z}|_{\LL^2} & \leq \; C\, h \, |A^{\frac{1}{2}}  {\bf z}|_{\LL^2}, \quad \forall {\bf z}\in V,   \label{I-Q_2} \\
|z-P^0_h z|_{L^2} & \leq C\, h\, |\nabla z|_{L^2}, \qquad \forall z\in W^{1,2}_{per}(D).		\label{I-P}
\end{align}

Using the Gagliardo-Nirenberg inequality \eqref{interpol}, we deduce from \eqref{I-Q_1} and \eqref{I-Q_2}  the following upper estimates for
${\bf z}\in V \cap \WW^{2,2}(D)$
\begin{equation} 		\label{I-Q_L4}
\|{\bf z}-\bQ {\bf z}\|_{\LL^4} \leq C h^{\frac{3}{2}} |A{\bf z}|_{\LL^2}, \quad {\rm and} \;  \|{\bf z}-\bQ {\bf z}\|_{\LL^4} \leq 
C h\, |A{\bf z}|_{\LL^2}^{\frac{1}{2}} \, |A^{\frac{1}{2}} {\bf z}|_{\LL^2}^{\frac{1}{2}}.
\end{equation}

 The following result about the pressure term will be used in the study of space-time discretization. It extends \cite[Lemma 3.2]{CarPro}
 to arbitrary moments. 
\begin{lemma} Let  \eqref{K0} be satisfied and suppose that $u_0$ is    ${\mathcal F}_0$-measurable such that   \eqref{gamma0} holds.  
Let  $\{\bu^l, \pi^l\}_{l=0, \cdots, N}$ be the solution to \eqref{full-imp1} and 
\eqref{fu-im-div}.
 Then    given  $q\in [2,\infty)$, there exists  a constant $C$ depending on $T,q,K_0$ and $\EE( \|u_0\|_V^{6q})<\infty$,  such that 
\begin{equation} 		\label{pressure-V}
\EE\Big(\Big|  \frac{T}{N} \sum_{l=1}^N |\nabla \pi^l |_{\LL^2}^2\Big|^q  \Big) \leq C, \qquad N=1, 2, ...
\end{equation}		
\end{lemma}
 \begin{proof} The beginning of  proof is similar to that in \cite{CarPro}; we include it for the sake of completeness. 

Let formally $\phi=\nabla \pi^l$  in \eqref{full-imp1};  since $\Delta_l W \in V$, we have  $\big( \Delta_l W, \nabla \pi^l\big)=0$.
Using \eqref{fu-im-div}, and then Cauchy-Schwarz and Young inequalities, we deduce that a.s. for $l=1, ..., N$ 
\begin{align} \label{est-press-G}
\frac{T}{N} |\nabla \pi^l|_{\LL^2}^2  \leq & \; \frac{T}{N}\, \nu\,  |\Delta \bu^l|_{\LL^2} \, |\nabla \pi^l|_{\LL^2} +  \frac{T}{N} | (\bu^l \, . \, \nabla ) \bu^l|_{\LL^2}
\, |\nabla \pi^l|_{\LL^2}  + | \Delta_l W|_{\LL^2} |\nabla \pi^l|_{\LL^2} \nonumber \\
\leq & \; \frac{T}{4N} |\nabla \pi^l|_{\LL^2}^2 + \frac{T}{N} \nu^2 |\Delta \bu^l|_{\LL^2}^2 + \frac{T}{4N} |\nabla \pi^l|_{\LL^2}^2 
+ \frac{T}{N}  C  | \bu^l |_{\LL^2} |A^{\frac{1}{2}} \bu^l|_{\LL^2}^2 |A \bu^l|_{\LL^2}^2, 
\end{align}
where the last upper estimate is deduced from the Gagliardo-Nirenberg inequality. 
Adding these upper estimates for $l=1$ to $N$,  taking expectation and  using once more Young's inequality,
 we obtain for every $q\in [1,\infty)$ 
\begin{align*}		
\EE\Big( \Big| \sum_{l=1}^N \frac{T}{2N} |\nabla \pi^l|_{\LL^2}^2 \Big|^q\Big)  \leq &\;  C(q,\nu )
 \EE\Big( \Big| \nu \frac{T}{N} \sum_{l=1}^N |A  \bu^l|_{\LL^2}^2\Big|^q\Big)
+ C(q) \EE\Big( \Big| \nu \frac{T}{N} \sum_{l=1}^N |A  \bu^l|_{\LL^2}^2\Big|^{2q}\Big)  \\ 
& + C(q) \EE\Big( \max_{1\leq l\leq N} \|\bu^l\|_V^{6q } \Big) .
\end{align*} 
Using 
\eqref{C(T,q)} and  \eqref{C(T,4)}, we deduce \eqref{pressure-V}. 
\end{proof}

\subsection{Description of the space-time schemes} 
We suppose that for some $q\in [1,+\infty)$ 
\begin{equation} 		\label{U0}
 \EE( |u_0 - \bU^0|_{\LL^2}^{2q} ) \leq C(q) \; h^{2q}, \qquad \EE(|A^{\frac{1}{2}} \bU^0|_{\LL^2}^{2q}) \leq C(q)
\end{equation}
for some positive constant $C(q)$. As it is usual in this framework, to ease notations we let $k:=\frac{T}{N}$ denote the constant time mesh. 
\\
For general finite elements satisfying the discrete LBB condition  \eqref{LBB}, one has to change the tri-linear term
$b(\bU_1, \bU_2,\bU_3 ) = ([\bU_1 \cdot \nabla] \bU_2\, , \bU_3)$ to control the nonlinear effect
 in the presence of discretely divergence-free velocity iterates,
and thus to allow stability of the scheme. Thus, we set 
\begin{equation}			\label{tildeb}
\tilde{b}(\bU_1, \bU_2, \bU_3) := ([\bU_1 \cdot \nabla ] \bU_2\, , \, \bU_3) + \frac{1}{2} \big( [{\rm div} \bU_1]\, \bU_2 \ , \bU_3\big), \quad
\forall \bU_1, \bU_2, \bU_3\in \WW^{1,2}.
\end{equation}
Note that this trilinear term is anti-symmetric with respect to the last two variables, i.e., 
\begin{equation}		\label{anti}
\tilde{b} (\bU_1, \bU_2, \bU_3) = - \tilde{b}(\bU_1, \bU_3, \bU_2), \qquad \forall \bU_1, \bU_2, \bU_3 \in \WW^{1,2}.
\end{equation}
Therefore, 
\[ 
\tilde{b}(\bU, \bPhi, \bPhi)=0, \qquad \forall \bU, \bPhi \in \WW^{1,2}.
\] 
\medskip

\noindent {\bf  Algorithm 1.}
 {\it Let $\bU^0$ be an  ${\mathcal F}_0$-measurable, $\HH_h$-valued random variable  and suppose that $W$ is a $V$-valued
 Brownian motion defined by \eqref{def_W}.  For every $l=1, \cdots, N$,
we consider a pair of $\HH_h\times L_h$ random variables $(\bU^l, \Pi^l)$ such that for every pair $(\bPhi,\Lambda)\in \HH_h\times L_h$,
 we have a.s. }
\begin{align}		\label{algo_FE}
(\bU^l - \bU^{l-1}, \bPhi) + k \, \nu \, (A^{\frac{1}{2}} \bU^l, A^{\frac{1}{2}} \bPhi) + k \big( [\bU^{l-1} \cdot \nabla] \bU^l, \bPhi\big)
&+ \frac{k}{2} \big( [{\rm div} \, \bU^{l-1}]\, \bU^l, \bPhi\big) \nonumber  \\
- k ( \Pi^l, {\rm div }\,  \bPhi) &= ( \Delta_l W, \bPhi),\\
({\rm div }\,  \bU^l, \Lambda)&=0. \label{divU}
\end{align}

The following result, which states the existence and uniqueness of the pairs ${(\bU^l, \Pi^l)}_{l=1, \cdots, N}$ and provides moments
of the solution, 
 has been proven in \cite[Lemma 3.1]{BrCaPr}  when $\bU^0$ is deterministic (see also \cite[Lemma 4.1]{CarPro} for a random initial condition).
  The exponents are dyadic
numbers because of an induction argument which enables to deduce results when doubling the exponent. 
\begin{lemma}		\label{moments_U}
Let  \eqref{K0} holds and 
let $\bU^0$ be an  ${\mathcal F}_0$-measurable, $\HH_h$-valued  random variable.
Suppose that  $\EE\big(\|\bU^0\|_{\LL^2}^{p}\big) \leq C(p)$ for some dyadic exponent $p=2^q\in [2,\infty)$ 
and some positive constant $C(p)$ independent of $h>0$. 
 Then for every $l=1, \cdots, N$, there exists a unique pair $(\bU^l, \Pi^l)$ of ${\mathcal F}_{l k}$-measurable, 
 $\VV_h\times L_h$-valued  random variables which satisfy \eqref{algo_FE}-\eqref{divU}. Furthermore,
\begin{align}
&\EE\Big( \max_{0\leq l\leq N} |\bU^l|_{\LL^2}^{p} + \nu\, k \sum_{l=1}^N |\bU^l|_{\LL^2}^{p-2} |A^{\frac{1}{2}} \bU^l|_{\LL^2}^2 \Big) \leq 
{\bf C}_1(T,p), 
   \label{U2}\\
&\EE \Big[ \Big( k \sum_{l=1}^N |A^{\frac{1}{2}} \bU^l|^2_{\LL^2}\Big)^{\frac{p}{2}}\Big] \leq {\bf C}_2(T,p), \label{sum_grad}
\end{align}
where the constants    ${\bf C}_i(T,q)$ 
 $i=1,2$ 
 do not depend on $N$ and $h>0$.
\end{lemma} 

We can reformulate the algorithm \eqref{algo_FE}-\eqref{divU} as follows, using divergence-free test  functions (see \cite[(4.4)]{CarPro}).\\
{\bf  Algorithm 2.}
{\it   Let $W$ and $\bU^0$ be as in {\bf Algorithm 1}. We have a.s. for $l=1, \cdots, N$ }
\begin{align}			\label{algo_special_FE}
(\bU^l - \bU^{l-1}, \bPhi) +  \nu \, k\,  (A^{\frac{1}{2}} \bU^l, A^{\frac{1}{2}} \bPhi) + &k \big( [\bU^{l-1} \cdot \nabla] \bU^l, \bPhi\big)
+ \frac{k}{2} \big( [{\rm div} \bU^{l-1}]\, \bU^l, \bPhi\big) \nonumber  \\
 &= \big(  \Delta_l W, \bPhi\big), \qquad \forall \bPhi \in \VV_h .
\end{align} 
As in \cite{CarPro}, we will compare the space-time scheme $\bU^l$ and the fully implicit time scheme $\bu^l$. 
For $l=0, \cdots, N$, let $\bE^l:=\bu^l- \bU^l$. 
 Note that using \eqref{divU} we have $\bU^l \in \VV^h$ for $l=1, ..., N$, so that 
\begin{equation}	\label{E-QE}
 \bE^l-\bQ \bE^l = \bu^l-\bQ \bu^l, \quad l=1, ...,N.
 \end{equation}

\subsection{Strong speed of convergence of the space-time scheme for general finite elements}
In this section we study the $L^2(\Omega)$-speed of convergence of the difference of the full space-time scheme, that is 
$\max_{0\leq l\leq N} |u(t_l) - \bU^l|_{\LL^2}^2$, in terms of
$k$ and $h$  when $\{ (\bU^l , \Pi^l) \}_l $  solves Algorithm 1. As in \cite{CarPro} and \cite{BeMi-FE}, we at first prove upper estimates of the $L^2(\Omega)$-norm
 of the difference between the time and full schemes, that is of $\max_{0 \leq l\leq N} |\bu^l-\bU^l|_{\LL^2}$. 

\begin{theorem} \label{E-E1}  
Let \eqref{K0} hold, $u_0\in V$ and $\bU^0$ be ${\mathcal F}_0$-measurable, taking values in 
$ \HH_h$, and such that  $\EE\big( \|\bU^0\|_{{\mathbb W}^{1,2}}^{q_0}\big)<\infty$ for some some $q_0$ large enough.
 Suppose furthermore that \eqref{U0} is satisfied for this value of $q_0$. 
 Let $\{(\bU^l, \Pi^l)\}_l$ be solution of  Algorithm 1.  

  Suppose that $ {\rm Tr}(Q) <\frac{4 \, \nu^3}{13 \,  [ T\nu  \,  \bar{C}^2+ 4 \, \sigma^2 ] }$.
Then,  for   $N$ large enough and  $h\in (0,1)$  we have  
\begin{equation}		\label{strong_gene}
\EE\Big( \max_{0\leq n\leq N}  |\bE^n|_{\LL^2}^2 + k \sum_{l=1}^N |A^{\frac{1}{2}} \bE^l|_{\LL^2}^2 \Big) \leq 
C  \big[ k+ h^2 \big]. 
\end{equation} 
\end{theorem}    
 \begin{remark}		\label{ex_Gaussian}
The assumption \eqref{U0} on the initial condition $U_0$ of Algorithm 1 holds for every $q\in [1,+\infty)$ if $U_0$ and $u_0-U_0$ are Gaussian random variables. 
\end{remark} 
\noindent {\it Proof of Theorem \ref{E-E1}}
 Parts of the proof are similar to Section 4.3 in \cite{BeMi-FE}; they are included for the sake of completeness. 
The proof is divided in several parts.
\smallskip

\noindent {\bf Part 1: A preliminary decomposition} \\
For every $l=1, \cdots, N$ and  $\bPhi\in \HH_h$, 
\begin{align*}
(\bE^l\!-\!\bE^{l-1},   \bPhi) &+ \nu k (A^{\frac{1}{2}} \bE^l  ,  A^{\frac{1}{2}} \bPhi) + k \tilde{b}(\bu^l, \bu^l, \bPhi) 
- k \tilde{b}( \bU^{l-1}, \bU^l, \bPhi) - k \big( \pi^l\!-\!\Pi^l , {\rm div }\, \bPhi\big) = 0.
\end{align*}

Since $\bU^l\in \VV_h$ for $l=1, \cdots, N$,  we have $\bQ \bU^l=\bU^l$; in the above identity,  choose 
\[ 
\bPhi : = \bQ \bE^l = \bE^l - \big( \bu^l-\bQ \bu^l\big).
\] 
Then, since $\bQ \bE^l\in \VV_h$ and $\Pi^l\in L_h$, using \eqref{Vh} we deduce that  $\big( \Pi^l, {\rm div}\, \bQ \bE^l\big)=0$ for $l=1, \cdots, N$. 
Since $(a, a-b) = \frac{1}{2} \big( |a|_{\LL^2}^2 - |b|_{\LL^2}^2 + |a-b|_{\LL^2}^2\big)$, we deduce
\begin{align}		\label{E^m-E^(m-1)}
\frac{1}{2} &\Big[ |\bQ \bE^l|_{\LL^2}^2 - |\bQ \bE^{l-1}|_{\LL^2}^2 + |\bQ (\bE^l - \bE^{l-1}|_{\LL^2}^2 \Big] 
+ \nu k |A^{\frac{1}{2}}  \bE^l|_{\LL^2}^2 
\nonumber \\
&\quad + k \tilde{b}\big(\bu^l-\bu^{l-1}, \bu^l , \bQ \bE^l\big) +k  \tilde{b}\big(\bu^{l-1}, \bu^l ,\bQ \bE^l\big) -
k  \tilde{b}\big(\bU^{l-1}, \bU^l,  
\bQ \bE^l\big) \nonumber \\
& = \nu k \big( A^{\frac{1}{2}} \bE^l, \nabla(\bu^l - \bQ \bu^l) \big) + k (\pi^l, {\rm div}\, \bQ \bE^l) 
\end{align}

 \noindent {\bf Part 2: Intermediate results}\\
 We prove upper estimates for the quantities involving the non-linear term $\tilde{b}$ and the terms in
the right handside of  \eqref{E^m-E^(m-1)}. Unlike \cite{BeMi-FE}, we prove a.s. estimates and no expected value is computed in this step. 
\smallskip

\noindent {\bf Part 2.1: Estimate of the error term $\big|\big( A^{\frac{1}{2}} \bE^l\, ,\,  A^{\frac{1}{2}} [\bu^l-\bQ \bu^l]\big)\big|$}\\
Using  
the Cauchy-Schwarz and  Young inequalities,  and then \eqref{I-Q_1}, we obtain for $\epsilon >0$
 \begin{align*}		
 \big|  \big( A^{\frac{1}{2}} \bE^l\, ,\,  A^{\frac{1}{2}} (\bu^l-\bQ \bu^l)\big)\big|\leq & \; \big| A^{\frac{1}{2}} \bE^l \big|_{\LL^2} \; 
 \big|A^{\frac{1}{2}} (\bu^l - \bQ \bu^l)\big|_{\LL^2}
  \leq  \; \epsilon \, \big| A^{\frac{1}{2}} \bE^l \big|_{\LL^2}^2 + C( \epsilon) \, h^2 \, |A\bu^l|_{\LL^2}^2. 
 \end{align*}
This implies for $\epsilon >0$ and $n=1, \cdots, N$ 
\begin{align}		\label{E1}
& 
 \nu  k \sum_{l=1}^n 
\big|  \big( A^{\frac{1}{2}} \bE^l ,\ A^{\frac{1}{2}} (\bu^l-\bQ \bu^l)\big)\big|  
\leq  \epsilon \nu  k \sum_{l=1}^n 
  |A^{\frac{1}{2}} \bE^l|_{\LL^2}^2  + C(\nu, \epsilon)   h^2 \Big(k
\sum_{l=1}^N \nu  |A \bu^l|_{\LL^2}^2 \Big).
\end{align}

\noindent {\bf Part 2.2: Estimate of the term $\tilde{b}(\bu^l-\bu^{l-1}, \, \bu^l, \, \bQ \bE^l)$}\\
Using the antisymmetry property 
 \eqref{anti}, the H\"older inequality
 and the inequality  $\|u\|_{\LL^4}\leq C\; | A^{\frac{1}{2}} u|_{\LL^2}$ coming from the Sobolev embedding theorem, we deduce
\begin{align*}
\big| \tilde{b}(\bu^l-&\bu^{l-1}, \, \bu^l, \, \bQ \bE^l) \big| 
 \leq \big| {b}(\bu^l-\bu^{l-1},  \bQ \bE^l, \, \bu^l ) \big| + \frac{1}{2} \big| \big( [{\rm div}\, (\bu^l - \bu^{l-1})] \; \bQ \bE^l\, , \, \bu^l\big) \big|  \\
 & \leq   \, \| \bu^l - \bu^{l-1}\|_{\LL^4}  \; |A^{\frac{1}{2}} \bQ \bE^l|_{\LL^2}\; \|\bu^l\|_{\LL^4}  + \frac{1}{2}  | A^{\frac{1}{2}} (\bu^l-\bu^{l-1}|_{\LL^2} 
\|\bQ \bE^l\|_{\LL^4}\; \| \bu^l\|_{\LL^4} \\
& \leq \; \epsilon\,  \nu\; |A^{\frac{1}{2}} \bE^l|^2_{\LL^2} + C(\nu, \epsilon) \| \bu^l - \bu^{l-1}\|^2_V\; \|\bu^l\|^2_V.
\end{align*}
Therefore, 
\begin{align}		\label{tildeb(u-u)}
k \sum_{l=1}^n \big| \tilde{b}(\bu^l-\bu^{l-1}, \, \bu^l, \, \bQ \bE^l) \big|  \leq & \;  
\epsilon\,  \nu\, k  \sum_{l=1}^n |A^{\frac{1}{2}} \bE^l|_{\LL^2}^2 \nonumber\\
&+ C(\nu, \epsilon) \, k \Big( \max_{1\leq l\leq N} \|\bu^l\|^2_V \Big) \sum_{l=1}^N\| \bu^l - \bu^{l-1}\|^2_V.
\end{align}
\medskip

\noindent {\bf Part 2.3: Estimate of the term $-\tilde{b}(\bu^{l-1}, \, \bu^l, \, \bQ \bE^l) + \tilde{b}(\bU^{l-1},\, \bU^l, \, \bQ \bE^l)$}

The antisymmetry property of 
$\tilde{b}$ in 
\eqref{anti}  implies for $l=1, \cdots, N$
\begin{align} 	\label{tildeb-tildeb}
-\tilde{b}(\bu^{l-1}, \, \bu^l, \, \bQ \bE^l) + \tilde{b}(\bU^{l-1}, \, \bU^l, \, \bQ \bE^l) = \sum_{i=1}^3 \tilde{T}_i(l),
\end{align} 
where 
\begin{align*}
\tilde{T}_1(l)=& \;  \tilde{b}(\bu^{l-1},\, \bE^l,\, \bE^l-\bQ \bE^l), \quad \tilde{T}_2(l)= \tilde{b}(\bE^{l-1},\, \bE^l,\, \bQ\bE^l - \bE^l), \\
\tilde{T}_3(l)=& \;   - \; \tilde{b}(\bE^{l-1},\, \bu^{l},\, \bQ\bE^l). 
\end{align*}

 {\it Terms $\tilde{T}_1(l)$.}
The identity  \eqref{E-QE}  
implies that $\tilde{T}_1(l)=-\tilde{b}(\bu^{l-1},\, \bu^l-\bQ\bu^l,\, \bE^l)$ for $l=1, ...,N$. Using the H\"older and
Gagliardo-Nirenberg inequalities,  \eqref{interpol} and \eqref{I-Q_1}, we deduce 
\begin{align}		\label{T1-b}
\big| b(\bu^{l-1},\, \bu^l - \bQ\bu^l,  \bE^l)\big| & \leq\;  \|\bu^{l-1}\|_{\LL^4} |A^{\frac{1}{2}}(\bu^l-\bQ\bu^l)|_{\LL^2} \| \bE^l\|_{\LL^4}  \nonumber \\
&\leq \; \sqrt{\bar{C}}  \, | A^{\frac{1}{2}} \bE^l|_{\LL^2}^{\frac{1}{2}} |\bE^l|_{\LL^2}^{\frac{1}{2}} \|\bu^{l-1}\|_V\;  h \; | A\bu^l|_{\LL^2} \nonumber\\
 &\leq \; \epsilon \nu |A^{\frac{1}{2}} \bE^l|_{\LL^2}^2 + 
 |\bE^l|_{\LL^2}^2  + 
 C( \nu, \epsilon) \, h^2\, \|\bu^{l-1}\|_V^2 |A\bu^l|_{\LL^2}^2, 
\end{align}
where the last upper estimate is deduced using Young's inequality. Furthermore, using once more the Gagliardo-Nirenberg inequality,
\eqref{I-Q_L4}  and then Young's inequality, we obtain
\begin{align*}
\big| \big( [{\rm div}\, \bu^{l-1}](\bu^l-\bQ\bu^l),\, \bE^l\big)\big| \leq &\; |A^{\frac{1}{2}} \bu^{l-1}|_{\LL^2} \, \|\bu^l-\bQ\bu^l\|_{\LL^4} \|\bE^l\|_{\LL^4} \\
\leq & \;  \sqrt{\bar{C}}  \,  |A^{\frac{1}{2}} \bE^l|_{\LL^2}^{\frac{1}{2}} |\bE^l|_{\LL^2}^{\frac{1}{2}} 
 |A^{\frac{1}{2}} \bu^{l-1}|_{\LL^2}
\; h^{\frac{3}{2}} |A\bu^l|_{\LL^2} \\
\leq & \; \epsilon \nu |A^{\frac{1}{2}} \bE^l|_{\LL^2}^2 + 
 |\bE^l|_{\LL^2}^2  + 
 C( \nu, \epsilon) \, h^3\, \|\bu^{l-1}\|_{V}^2 |A\bu^l|_{\LL^2}^2. 
\end{align*}

Therefore,  if $h\in (0,1)$,   for $n=1, \cdots, N$
\begin{align}		\label{tildeT1-FEM}
 k \sum_{l=1}^n  |\tilde{T}_1(l)| \leq  &
\; 2\epsilon \nu  k \sum_{l=1}^n  |A^{\frac{1}{2}} \bE^l|_{\LL^2}^2  + 2 
  k \sum_{l=1}^n |\bE^l|_{\LL^2}^2 \nonumber \\
&  \quad  + C(\nu, \epsilon)\,   h^2 \,  \max_{0\leq l\leq N} \|\bu^{l}\|_V^2 
 \Big( \nu\,  k \sum_{l=1}^N   |A\bu^l|_{\LL^2}^2  \Big). 
\end{align}
\smallskip

{\it Terms $\tilde{T}_2(l)$.} Using once more \eqref{E-QE},  we  replace the difference $\bE^l-\bQ \bE^l$ by
$\bu^l-\bQ \bu^l$ for $l=1, ...,N$.
The Gagliardo-Nirenberg inequality, \eqref{I-Q_L4} and Young's inequality imply
\begin{align}	\label{b-tildeT2}
& \big| b(\bE^{l-1}, \, \bE^l,\, \bu^l-\bQ\bu^l)\big| \leq \; \sqrt{\bar{C}} |A^{\frac{1}{2}} \bE^{l-1}|_{\LL^2}^{\frac{1}{2}} |\bE^{l-1}|_{\LL^2}^{\frac{1}{2}} 
|A^{\frac{1}{2}} \bE^l|_{\LL^2} \|\bu^l-\bQ\bu^l\|_{\LL^4} \nonumber \\
&\; \, \leq \epsilon \nu |A^{\frac{1}{2}} \bE^{l-1}|_{\LL^2}^2 + \epsilon \nu |A^{\frac{1}{2}} \bE^l|_{\LL^2}^2 + C(\nu,\epsilon) \, h^4 \, |\bE^{l-1}|_{\LL^2}^2 
|A\bu^l|_{\LL^2}^2 |A^{\frac{1}{2}} \bu^l|_{\LL^2}^2\nonumber \\
&\; \, \leq  \epsilon \nu |A^{\frac{1}{2}} \bE^{l-1}|_{\LL^2}^2 + \epsilon \nu |A^{\frac{1}{2}} \bE^l|_{\LL^2}^2 + C(\nu,\epsilon) \, h^4 \big( |\bu^{l-1}|_{\LL^2}^2
+ |\bU^{l-1}|_{\LL^2}^2\big) |A\bu^l|_{\LL^2}^2 |A^{\frac{1}{2}} \bu^l|_{\LL^2}^2.
\end{align}
Furthermore, the H\"older and Gagliardo-Nirenberg inequalities together with \eqref{I-Q_L4} and Young's inequality yield 
\begin{align*}	
\big| \big( [{\rm div}&\, \bE^{l-1}] \, \bE^l,\, \bu^l-\bQ\bu^l\big)\big| \leq  \; \sqrt{\bar{C}}  |A^{\frac{1}{2}} \bE^{l-1}|_{\LL^2} |A^{\frac{1}{2}} \bE^l|_{\LL^2}^{\frac{1}{2}}
|\bE^l|_{\LL^2}^{\frac{1}{2}} \| \bu^l-\bQ\bu^l\|_{\LL^4} \\ 
\leq & \; \epsilon \nu |A^{\frac{1}{2}} \bE^{l-1}|_{\LL^2}^2 + \epsilon \nu |A^{\frac{1}{2}} \bE^l|_{\LL^2}^2 + C(\nu,\epsilon) \, h^4 \, |\bE^{l}|_{\LL^2}^2 
|A\bu^l|_{\LL^2}^2 |A^{\frac{1}{2}} \bu^l|_{\LL^2}^2 \\ 
\leq &\; \; \epsilon \nu |A^{\frac{1}{2}} \bE^{l-1}|_{\LL^2}^2 + \epsilon \nu |A^{\frac{1}{2}} \bE^l|_{\LL^2}^2 + C(\nu,\epsilon) \, h^4 \big( |\bu^{l}|_{\LL^2}^2
+ |\bU^{l}|_{\LL^2}^2\big) |A\bu^l|_{\LL^2}^2 |A^{\frac{1}{2}} \bu^l|_{\LL^2}^2.
\end{align*} 
This inequality and  \eqref{b-tildeT2}  yield for $n=1, ..., N$
\begin{align}			\label{tildeT2-FEM}
& k \sum_{l=1}^n  |\tilde{T}_2(l)| \leq \; 2 \epsilon \nu k\, |A^{\frac{1}{2}} \bE^0|_{\LL^2}^2
+ 4 \epsilon \nu k  \sum_{l=1}^n |A^{\frac{1}{2}} \bE^l|_{\LL^2}^2 \nonumber  \\
&\qquad + C(\nu, \epsilon) \, h^4\,  \Big(\max_{0\leq l\leq N}  \|\bu^l\|_V^4 + \max_{0\leq l\leq N} |\bU^l|_{\LL^2}^4\big] \Big) \; 
\Big( \nu\, k \sum_{l=1}^N |A\bu^l|^2_{\LL^2}\Big). 
\end{align}
\smallskip

{\it Term $\tilde{T}_3(l)$.}  
The Gagliardo-Nirenberg inequality \eqref{interpol} implies
\[ \big| b(\bE^{l-1},\, \bu^{l},\, \bQ\bE^l)\big| \leq \bar{C} |A^{\frac{1}{2}} \bE^{l-1}|_{\LL^2}^{\frac{1}{2}} |\bE^{l-1}|_{\LL^2}^{\frac{1}{2}} 
|A^{\frac{1}{2}} \bu^{l}|_{\LL^2} |A^{\frac{1}{2}} \bE^l|_{\LL^2}^{\frac{1}{2}} |\bE^l|_{\LL^2}^{\frac{1}{2}}.
\]
Using \eqref{interpol}, the H\"older and Young inequalities we obtain for $\epsilon>0$,  $l=1, ...,N$ 
\begin{align*}	
\big| b(&\bE^{l-1},\, \bu^{l},\, \bQ\bE^l)\big| \leq \,   \bar{C}  |A^{\frac{1}{2}} \bE^{l-1}|_{\LL^2}^{\frac{1}{2}} |A^{\frac{1}{2}} \bE^l|_{\LL^2}^{\frac{1}{2}} 
\Big(  |A^{\frac{1}{2}} \bu^{l}|_{\LL^2}^{\frac{1}{2}}  |\bE^l|_{\LL^2}^{\frac{1}{2}}\Big) 
\Big(   |A^{\frac{1}{2}} \bu^{l}|_{\LL^2}^{\frac{1}{2}}  |\bE^{l-1}|_{\LL^2}^{\frac{1}{2}}\Big) \\
\leq & \; \epsilon \nu |A^{\frac{1}{2}} \bE^{l-1}|_{\LL^2}^2 + \epsilon \nu |A^{\frac{1}{2}} \bE^{l}|_{\LL^2}^2 
+ \frac{  \bar{C}^2  }{2^4 \epsilon \nu}  |A^{\frac{1}{2}} \bu^{l}|_{\LL^2}^{2} 
|\bE^l|_{\LL^2}^2 + \frac{ \bar{C}^2 }{2^4 \epsilon \nu}  |A^{\frac{1}{2}} \bu^{l}|_{\LL^2}^{2}   |\bE^{l-1}|_{\LL^2}^2 .
\end{align*} 
This yields for $n=1, \cdots, N$ 
\begin{align}			\label{T3-1-N}
 k \sum_{l=1}^n  \big| b(\bE^{l-1},\, \bu^{l},\, \bQ\bE^l &)\big|  \leq   2 \epsilon  \nu k \sum_{l=1}^n 
|A^{\frac{1}{2}} \bE^l|_{\LL^2}^2  +  \frac{  \bar{C}^2 }{2^3 \epsilon \nu}  \Big(\max_{1\leq l\leq N}  
 |A^{\frac{1}{2}} \bu^{l}|_{\LL^2}^{2} \Big)  \,  k \sum_{l=1}^{n-1}
|\bE^l|_{\LL^2}^2 
 \nonumber \\
 &+ \epsilon \nu k |A^{\frac{1}{2}} \bE^0|_{\LL^2}^2 +
 C(\nu,\epsilon) \, k\,  \Big[ \max_{1\leq l\leq N} \|\bu^l\|_V^4  
 + \max_{1\leq l\leq N }|\bU^l|_{\LL^2}^4\Big]. 
\end{align}
\smallskip
On the other hand, $\big( [{\rm div}\, \bE^{l-1}]\, \bu^{l},\, \bQ\bE^l\big)=\tilde{T}_4(l) + \tilde{T}_5(l)$, where 
\[ 
\tilde{T}_4(l): = \big( [{\rm div}\, \bE^{l-1}]\, \bu^{l},\, \bQ(\bE^l-\bE^{l-1})\big) \quad \mbox{\rm and } \quad 
\tilde{T}_5(l):= \big( [{\rm div}\, \bE^{l-1}]\, \bu^{l},\, \bQ\bE^{l-1}\big).
\] 
The H\"older and Gagliardo-Nirenberg inequalities \eqref{interpol} imply  
\begin{align*}
|\tilde{T}_4(l)| &\leq \;  \sqrt{\bar{C}} \; |A^{\frac{1}{2}} \bE^{l-1}|_{\LL^2}
\| \bu^{l}\|_{\LL^4}   |A^{\frac{1}{2}} \bQ(\bE^{l}-\bE^{l-1})|_{\LL^2}^{\frac{1}{2}}  |\bQ(\bE^{l}-\bE^{l-1})|_{\LL^2}^{\frac{1}{2}}\\
&\leq \;  \frac{ \bar{C}}{\sqrt{2}  }\;  |A^{\frac{1}{2}} \bE^{l-1}|_{\LL^2}
\| \bu^{l}\|_{V}   |A^{\frac{1}{2}} (\bE^{l}-\bE^{l-1})|_{\LL^2}^{\frac{1}{2}}  |\bQ(\bE^{l}-\bE^{l-1})|_{\LL^2}^{\frac{1}{2}}.
\end{align*}
The Young inequality implies that for any $\epsilon, \epsilon_1>0$, we have for $n=1, \cdots, N$
\begin{align}		\label{tildeT4-FEM}
 k \sum_{l=1}^n  \big| \tilde{T}_4(l)\big| \leq &
\; \epsilon \nu k \sum_{l=1}^n |A^{\frac{1}{2}} \bE^l|_{\LL^2}^2 
+\epsilon_1 \sum_{l=1}^n  |\bQ(\bE^{l}-\bE^{l-1})|_{\LL^2}^2    \nonumber \\
& \quad +  C  \frac{\bar{C}^4}{\nu^2\, \epsilon^2\, \epsilon_1}    k^2 \sum_{l=1}^n \| \bu^{l}\|_{V}^4    
\Big( |A^{\frac{1}{2}} (\bU^{l}-\bU^{l-1})|_{\LL^2}^2+ |A^{\frac{1}{2}} (\bu^{l}-\bu^{l-1})|_{\LL^2}^2\Big) \nonumber \\
 \leq & \; \epsilon \nu k \sum_{l=1}^n |A^{\frac{1}{2}} \bE^l|_{\LL^2}^2  +\epsilon_1 \sum_{l=1}^n  |\bQ(\bE^{l}-\bE^{l-1})|_{\LL^2}^2  \nonumber \\
&  \quad + C(\nu, \bar{C}, \epsilon, \epsilon_1, T)\, k\, \Big[1 +  \max_{1\leq l\leq N} \| \bu^l\|_V^8 +  \Big( k\sum_{l=1}^N |A^{\frac{1}{2}} \bU^l|_{\LL^2}^2\Big)^2\Big].
\end{align}
\smallskip

Furthermore, the Cauchy-Schwarz inequality and the Sobolev embedding theorem   yield
\begin{align*}
|\tilde{T}_5(l)| &\leq |A^{\frac{1}{2}} \bE^{l-1}|_{\LL^2}|\bu^{l}|_{\LL^\infty}|\bQ\bE^{l-1}|_{\LL^2} 
\leq \sigma \, |A^{\frac{1}{2}} \bE^{l-1}|_{\LL^2}\, |A\bu^{l}|_{\LL^2}\, |\bE^{l-1}|_{\LL^2}, 
\end{align*} 
where $\sigma$ is the constant defined in \eqref{Sobolev}. 
Therefore,  given $\epsilon >0$, the Young inequality implies 
\begin{align}\label{tildeT5-FEM}
k \sum_{l=1}^n \big| \tilde{T}_5(l)\big| 
\leq \epsilon\nu k \sum_{l=1}^n  |A^{\frac{1}{2}} \bE^{l-1}|_{\LL^2}^2
+  \frac{\sigma^2}{4\epsilon \nu}  k \sum_{l=0}^{n-1}   |A\bu^{l+1}|_{\LL^2}^2\, |\bE^{l}|_{\LL^2}^2. 
\end{align}

\noindent {\bf Part  2.4: Estimate of $(\pi^l, {\rm div}\, \bE^l)$} \\
Since $\bQ \bE^l\in \VV_h$ and $P^0_h \pi^l\in L_h$, using \eqref{Vh} we deduce 
$(\pi^l , {\rm div}\, \bQ \bE^l) = (\pi^l- P^0_l \pi^l, {\rm div}\, \bQ \bE^l)$.
Therefore, the Cauchy-Schwarz and Young inequalities  coupled with  \eqref{I-P} 
 imply
\begin{align}		\label{error-pressure-W}
 k  \sum_{l=1}^n  \big| \big( \pi^l, \, {\rm div}\, \bQ \bE^l\big)\big|&\leq  \; 
   \epsilon \nu k \sum_{l=1}^n \,  |A^{\frac{1}{2}} \bE^l|_{\LL^2}^2 
 + \frac{1}{4\epsilon \nu} k \sum_{l=1}^n |\pi^l - P^0_h \pi^l|_{L^2}^2 \nonumber \\
& \leq  \epsilon \nu k \sum_{l=1}^n \,  |A^{\frac{1}{2}} \bE^l|_{\LL^2}^2 + 
 C(\nu, \epsilon )  h^2 k \sum_{l=1}^N  |A^{\frac{1}{2}} \pi^l|_{L^2}^2. 
\end{align}

\noindent {\bf Part 3: The convergence result} \\
Summing over $l$ in \eqref{E^m-E^(m-1)} and collecting  the  estimates \eqref{E1}--\eqref{error-pressure-W} proved in Part 2, given 
 $\epsilon, \epsilon_1 >0$, 
  we deduce for all $n=1, ..., N$ and $h\in (0,1)$ 
\begin{align}		\label{Q0h-E^nt}
\frac{1}{2}  |\bQ \bE^n|_{\LL^2}^2& +\Big(\frac{1}{2}-\epsilon_1\Big)\sum_{l=1}^n |\bQ (\bE^l - \bE^{l-1})|_{\LL^2}^2 
+ \nu k (1 -13\epsilon) \sum_{l=1}^n |A^{\frac{1}{2}}  \bE^l|_{\LL^2}^2 \nonumber \\
\leq  & \;  k\; \sum_{l=0}^{n-1} \Big[ 2  + \frac{\bar{C}^2}{8 \epsilon \nu} \max_{1\leq j \leq N} 
 |A^{\frac{1}{2}} \bu^{j}|_{\LL^2}^{2} 
+ \frac{ \sigma^2}{4\epsilon \nu} |A \bu^{l+1}|_{\LL^2}^2
 \Big] |\bE^l |_{\LL^2}^2  + \frac{1}{2}  |\bQ \bE^0|_{\LL^2}^2  \nonumber \\
 &\; + C(\nu, \epsilon)  \, k\,   \big[ \|u^0\|_V^2 + |A^{\frac{1}{2}} \bU^0  |_{\LL^2}^2\big]  
 +4\, k\,    \max_{1\leq l\leq N}    \big[ |\bu^l |_{\LL^2}^2 + |\bU^l |_{\LL^2}^2  \big]   \nonumber \\
&+ C(\nu, \epsilon) \, k  \, \Big(\max_{1\leq l\leq N} \|\bu^l\|^2_V \Big)  \Big( \sum_{l=1}^N\| \bu^l - \bu^{l-1}\|^2_V \Big)
 + C(\nu, \epsilon )  h^2 \Big( k \sum_{l=1}^N  |\nabla \pi^l|_{L^2}^2\Big)  \nonumber\\
&+ C(\nu,\epsilon) h^2 \Big( 1+ \max_{1\leq l\leq N} \|\bu^l\|_V^4 + \max_{1\leq l\leq N} |\bU^l |_{\LL^2}^4\Big)
\Big[ 1+  \Big( \nu\, k\, \sum_{l=1}^N |A\bu^l |_{\LL^2}^2\Big) \Big] \nonumber \\
&+ C(\nu, \epsilon, \epsilon_1) \, k \,  \Big[ 1+ \max_{1\leq l\leq N} \|\bu^l\|_V^8 + \Big( k\sum_{l=1}^N |A^{\frac{1}{2}} \bU^l|_{\LL^2}^2\Big) \Big] .
\end{align} 
On the other side, using \eqref{I-Q_2} we deduce that for any $\lambda >0$
\begin{align}		\label{bE-Q0bE} 
|\bE^n|_{\LL^2}^2&\leq (1+\lambda) |\bQ \bE^n|_{\LL^2}^2 +  \Big( 1+\frac{1}{\lambda}\Big) | \bE^n-\bQ \bE^n|_{\LL^2}^2 \nonumber \\
&\leq (1+\lambda) |\bQ \bE^n|_{\LL^2}^2 +   \Big( 1+\frac{1}{\lambda}\Big) | \bu^n-\bQ \bu^n|_{\LL^2}^2 \nonumber  \\
& \leq (1+\lambda) |\bQ \bE^n|_{\LL^2}^2 + C(\lambda) \, h^2 \,  |A^{\frac{1}{2}} \bu^l|_{\LL^2}^2  .
\end{align}

For $l=0, \cdots, N-1$ set 
\begin{equation}\label{gl}
g_l =: k \, 2\, (1+\lambda) \Big[ 2  + \frac{\bar{C}^2}{8 \epsilon \nu} \max_{1\leq j \leq N} 
 |A^{\frac{1}{2}} \bu^{j}|_{\LL^2}^{2}  +\frac{ \sigma^2} {4\epsilon \nu} |A \bu^{l+1}|_{\LL^2}^2 \Big]
\end{equation}
and 
\begin{align}\label{Zt}
\tilde{Z} := &\;  (1+\lambda)   | \bE^0|_{\LL^2}^2 
+ C(\nu, \epsilon, \lambda )  \, k\,   \big[ \|u^0\|_V^2 + |A^{\frac{1}{2}} \bU^0  |_{\LL^2}^2\big]  
 +C(\nu, \epsilon, \lambda)\,   k    \max_{1\leq l\leq N}    \big[ |\bu^l |_{\LL^2}^4 + |\bU^l |_{\LL^2}^4  \big]   \nonumber \\
&\; + C(\nu, \epsilon, \lambda ) \, k  \, \Big(\max_{1\leq l\leq N} \|\bu^l\|^2_V \Big)  \Big( \sum_{l=1}^N\| \bu^l - \bu^{l-1}\|^2_V \Big)
 + C(\nu, \epsilon ,\lambda ) \,  h^2 \Big( k \sum_{l=1}^N  |\nabla \pi^l|_{L^2}^2\Big)  \nonumber\\
&\; + C(\nu,\epsilon,\lambda ) \, h^2 \, \Big( 1+ \max_{1\leq l\leq N} \|\bu^l\|_V^4 + \max_{1\leq l\leq N} |\bU^l |_{\LL^2}^4\Big)
\Big[ 1+  \Big( \nu\, k\, \sum_{l=1}^N |A\bu^l |_{\LL^2}^2\Big) \Big] \nonumber \\
&\; + C(\nu, \epsilon, \epsilon_1,\lambda) \, k \,  \Big[ 1+ \max_{1\leq l\leq N} \|\bu^l\|_V^8 + \Big( k\sum_{l=1}^N |A^{\frac{1}{2}} \bU^l|_{\LL^2}^2\Big) \Big] 
\nonumber \\ &\; 
+ C(\lambda) \, h^2 \, \max_{1\leq l\leq N} |A^{\frac{1}{2}} \bu^l|_{\LL^2}^2 .
\end{align}
For $n=0, ..., N$ set
\begin{equation} 		\label{def_Xn}
X_n:= |\bE^n|_{\LL^2}^2 + k \, \nu\, (1-13\epsilon) \sum_{l=1}^n |A^{\frac{1}{2}} \bE^l|_{\LL^2}^2
\end{equation}
using the convention $\sum_{l=0}^L y_l=0$ if $L<0$ for any sequence $\{y_l\}_l$ of non-negative numbers. 
Combining  the estimates  \eqref{Q0h-E^nt}--\eqref{Zt} and 
 we infer 
for $\epsilon_1 = \frac{1}{2}$ and $\epsilon \in \big(0, \frac{1}{13}\big)$  
\begin{align*}
0\leq X_n\leq \Zt+ \sum_{l=0}^{n-1} g_{l}  X_l,  \qquad n=0, ..., N. 
\end{align*}
The Special Gronwall inequality proved in \cite[s5]{Holte}  implies that for $n=0, ..., N$ 
\begin{align*}
X_n \leq \Zt \, {\rm e}^{ \sum_{l=0}^{n-1} g_{l}} \leq  \Zt \, {\rm e}^{ \sum_{l=0}^{N-1} g_{l}}.
\end{align*}
Therefore, 
\[ 
\max_{0\leq n\leq N} 
X_n \leq \Zt \, {\rm e}^{ \sum_{l=0}^{N-1} g_{l}}  , 
\] 
and the H\"older inequality implies that for $p, q\in (1,\infty)$ with  $\frac{1}{p}+\frac{1}{q}=1$
\begin{align}		\label{upper_Xn}
\EE\Big( \max_{0\leq n\leq N} 
X_n  \Big)\leq \EE\Big\{ \Big[\big| \Zt\big|^q \Big]\Big\}^{1/q} \, 
\Big\{ \EE \Big[ \exp\Big( p \sum_{l=0}^{N-1} g_{l}\Big) \Big]\Big\}^{1/p} .
\end{align}
Note that 

\begin{equation}  \label{sum-gl} 
 \sum_{l=0}^{N-1} g_l = \;  4T(1+\lambda) + \frac{\bar{C}^2  (1+\lambda) T}{4\epsilon  \nu }  \max_{1\leq l\leq N}  |A^{\frac{1}{2}} \bu^{l}|_{\LL^2}^{2}  
 + \frac{ \sigma^2  (1+\lambda)}{2\epsilon  \nu^2}\, \frac{T}{N} \nu \sum_{l=1}^N |A\bu^l|_{\LL^2}^2. 
\end{equation} 

 The assumption on ${\rm Tr}(Q)$ implies 
$ \frac{13\,  {\rm Tr}(Q)}{ \nu^2} \Big( \frac{T\, \bar{C}^2 }{4}  + \frac{ \sigma^2}{\nu}\Big) <1$.  
Therefore, we may choose $\lambda\in (0,1)$ (close to 0) and $\epsilon \in \big( 0, \frac{1}{13})$ (close to $\frac{1}{13}$) such that
\[  \frac{ (1+\lambda)^2\, {\rm Tr}(Q)}{ \epsilon \nu^2} \Big( \frac{T \, \bar{C}^2 }{4}  + \frac{ \sigma^2}{\nu}\Big) <1.\]
For $\tilde{\alpha}_0$ defined by \eqref{alpha0},  set  $p_1:= \frac{4\, \nu^2\, \epsilon}{(1+\lambda)^2\, \bar{C}^2\,  T \, {\rm Tr}(Q)}= 
\frac{4\, \epsilon\, \nu }{\bar{C}^2 T (1+\lambda)^2}\,  \tilde{\alpha_0}$ 
and $p_2:= \frac{ \nu^3\, \epsilon}{ (1+\lambda)^2\, \sigma^2\,   \, {\rm Tr}(Q)}=  \frac{2\nu^2 \epsilon}{\sigma^2 (1+\lambda)^2} \, \frac{\tilde{\alpha}_0}{2}$.
 Then $p_1\,   \frac{\bar{C}^2 (1+\lambda)\, T}{4\, \epsilon\, \nu} < \tilde{\alpha}_0$ and 
 $p_2\, \frac{\sigma^2 (1+\lambda)}{\epsilon \, \nu^2} < \frac{\tilde{\alpha}_0}{2}$. Furthermore, the choice of $\lambda$ and $\epsilon$ ensures that
 $ \frac{1}{p}:= \frac{1}{p_1} + \frac{1}{p_2}<1$.

  The H\"older inequality and the upper estimates \eqref{sum-gl}, \eqref{exp-moments-ul_al} and \eqref{exp-moments-Aul_al} imply 
  \begin{align*}
  \Big\{ \EE\Big[ \exp\Big( p \sum_{l=0}^{N-1} g_l\Big) \Big] \Big\}^{\frac{1}{p}} \leq&\;   e^{4T(1+\lambda)} \,
   \Big\{ \EE\Big[ \exp\Big(p_1  \frac{(1+\lambda) \bar{C}^2 T}{4\epsilon \nu} \max_{1\leq l\leq N} 
  |A^{\frac{1}{2} \bu^{l}|_{\LL^2}^{2}} \Big) \Big] \Big\}^{\frac{1}{p_1}}\\
  & \Big\{ \EE\Big[ \exp\Big(p_2  \frac{(1+\lambda) \bar{C}^2 }{2\epsilon \nu^2} \frac{T}{N} \nu \sum_{1\leq l\leq N} |A \bu^l|_{\LL^2}^2\Big) \Big] \Big\}^{\frac{1}{p_2}}
  <\infty. 
  \end{align*} 

We next provide upper estimates of $\EE(|\tilde{Z}|^q)$, where  $p, \epsilon$ and $\lambda$ have been chosen above, and $q$ is the conjugate exponent of $p$.  
The definition of $\tilde{Z}$ given in \eqref{Zt} and the Young inequality imply  for $h\in (0,1)$ 
\begin{align*}
\EE(|\tilde{Z}|^q)\leq & \; C\,  \EE\big(|\bE^0|_{\LL^2}^{2q}\big) + C \, k^q\Big\{  \EE\big( |A^{\frac{1}{2}} \bU^0|_{\LL^2}^{2q}\big) +  1
 + \EE\Big( \max_{0\leq l\leq N} \|\bu^l\|_V^{8q} \Big) 
+ \EE\Big( \max_{1\leq l\leq N} |\bU^l|_{\LL^2}^{4q} \Big) \\
& 
\qquad + \EE\Big( \Big| \sum_{l=1}^N \| \bu^l-\bu^{l-1}\|_V^2 \Big|^{ 2q}\Big)  + \EE\Big( \Big| k \sum_{l=1}^N |A^{\frac{1}{2}} \bU|_{\LL^2}^2\Big|^q\Big) \Big\} \\
&+ C h^{2q}  \Big\{ \EE\Big( \Big| k \sum_{l=1}^N |\nabla \pi^l|_{\LL^2}^2\Big|^q\Big) + 1+\EE\Big( \max_{1\leq l\leq N} \|\bu^l\|_V^{8q} \Big)
+ \EE\Big( \max_{1\leq l\leq N} |\bU^l|_{\LL^2}^{8q}\Big) \\
&\qquad + \EE\Big( \Big| \nu k \sum_{l=1}^N |A\bu^l|_{\LL^2}^2\Big|^{2q} \Big) \Big\}. 
\end{align*}
Therefore, the inequalities \eqref{U0}, \eqref{C(T,q)}, \eqref{C(T,4)},  \eqref{U2} and \eqref{sum_grad}  imply that for $q=  2^{q_0-2}$ for $q_0$ large enough,
 we have 
\[ 
\EE\Big( \max_{0\leq n\leq N} |\bE^n|_{\LL^2}^2\Big) \leq C  \Big[ k + h^2 
+ h^2 \Big\{ \EE\Big( \Big| k \sum_{l=1}^N |\nabla \pi^l|_{\LL^2}^2\Big|^q\Big)\Big\}^{\frac{1}{q}} \Big]. 
\] 
Since $\|u_0\|_V^2$ has exponential moments, using \eqref{pressure-V} we obtain 
\begin{equation}			\label{first-speed-E}
\EE\Big( \max_{0\leq n\leq N} X_n 
\Big) \leq C  \Big( k + h^2   \Big). 
\end{equation} 
Since $\max_{0\leq n\leq N} |\bE^n|_{\LL^2}^2 \leq \max_{0\leq n\leq N} X_n$ and $k \sum_{l=1}^N |A^{\frac{1}{2}} \bE^l|_{\LL^2}^2 \leq \frac{1}{\nu (1-13\epsilon)}
\max_{0\leq n\leq N} X_n$,
we deduce \eqref{strong_gene}; this completes the proof. 
\hfill $\square$

The following $L^2(\Omega)$-speed of convergence  is the main result; it is a straightforward consequence of Theorems \ref{Loc_cv_Euler} and \ref{E-E1}. 
The conditions on ${\rm Tr}(Q)$ in Theorem \ref{E-E1} are stronger than that in Theorem \ref{Loc_cv_Euler}.
\begin{theorem}	\label{cor_FEM_gene}
Let  $u_0$, $\bU^0$ and  ${\rm Tr}(Q) $ be as in Theorem \ref{E-E1}. Then 
given any  $\eta \in (0,1) $,   we have for $N$ large enough, $k=T/N$ and $h\in (0,1)$
\begin{equation}		\label{strong_gene_2}
\EE \Big( \max_{0\leq l\leq N} |u(t_l)-\bU^l|_{\LL^2}^2 + \frac{T}{N} \sum_{l=1}^N \big|A^{\frac{1}{2}} (u(t_l)) - A^{\frac{1}{2}} \bU^l\big|_{\LL^2}^2 \Big)
\leq  C\big[ k^\eta + h^2 \big]. 
\end{equation}
\end{theorem} 

\subsection{Strong convergence of divergence-free finite elements}
In this section we suppose that the finite elements are divergence free, that is ${\bf V}_h \subset V$ and consider the solution $\{\bU^l\}_l$ of Algorithm 2.

We at first prove the $L^2(\Omega)$ speed of convergence of $\max_{0\leq l\leq N} |\bE^l|_{\LL^2}$ where $\bE^l= \bu^l - \bU^l$. 
\begin{theorem} \label{E-E2}
Let \eqref{K0} hold;  suppose that ${\rm Tr}\, (Q) < \frac{4\, \nu^2}{5  \bar{C}^2\,   T}$ and 
let  $\bU^0\in L^{q_0}(\Omega; \HH_h)$  for some $q_0$ large enough  be  an 
${\mathcal F}_0$-measurable random variable  such that \eqref{U0} is satisfied for $q_0$. Let 
  $\{\bU^l\}_l$ be the solution of  Algorithm 2. 
  Then,  for   $N$ large enough and  $h\in (0,1)$ we have 
  \begin{equation}		\label{strong_div-free}
\EE\Big( \max_{0\leq n\leq N}  |\bE^n|_{\LL^2}^2 +  k \sum_{l=1}^N |A^{\frac{1}{2}}  \bE^l|_{\LL^2}^2 \Big) \leq 
  C\,  \big( k+ h^2 \big). 
\end{equation} 
 \end{theorem}  
\begin{proof}
We briefly sketch the argument which is similar to that used to prove Theorem \ref{E-E1}. 

Since in this case   $\tilde{b}=b $ and ${\rm div}\, \bQ \bE^l =0$, the identity  \eqref{E^m-E^(m-1)}  can be rewritten 
\begin{align}		\label{diff-E-Bis}
\frac{1}{2} &\Big[ |\bQ \bE^l|_{\LL^2}^2 - |\bQ \bE^{l-1}|_{\LL^2}^2 + |\bQ (\bE^l - \bE^{l-1}|_{\LL^2}^2 \Big] 
+ \nu k |A^{\frac{1}{2}}  \bE^l|_{\LL^2}^2 
\nonumber \\
&\quad + k {b}\big(\bu^l-\bu^{l-1}, \bu^l , \bQ \bE^l\big) +k  {b}\big(\bu^{l-1}, \bu^l ,\bQ \bE^l\big) -
k  {b}\big(\bU^{l-1}, \bU^l,  
\bQ \bE^l\big) \nonumber \\
& = \nu k \big( A^{\frac{1}{2}}  \bE^l, A^{\frac{1}{2}} (\bu^l - \bQ \bu^l) \big) . 
\end{align} 
In the upper estimate \eqref{tildeb(u-u)} the upper bound in the right hand side is similar if in the left hand side we replace $\tilde{b}$ by $b$.

Let $T_i(l)$, $i=1,2,3$ and $l=1, ..., N$ denote the quantities defined in \eqref{tildeb-tildeb} where $\tilde{b}$ is replaced by $b$. 
Adding the upper estimates \eqref{T1-b}  and \eqref{b-tildeT2},  
and using the upper estimates \eqref{diff-E-Bis},  \eqref{E1}, \eqref{tildeb(u-u)} with $b$ instead of $\tilde{b}$, \eqref{T3-1-N}, we deduce that  
for $n=1, ..., N$ and $h\in (0,1)$ 
\begin{align}		\label{Q0h-E^nt-divfree}
\frac{1}{2} & |\bQ \bE^n|_{\LL^2}^2 
+ \nu k (1 -5\epsilon) \sum_{l=1}^n |\nabla  \bE^l|_{\LL^2}^2 
\leq  k\; \sum_{l=0}^{n-1} \Big[ 1  + \frac{\bar{C}^2}{8 \epsilon \nu} \max_{1\leq j \leq N} 
 |A^{\frac{1}{2}} \bu^{l}|_{\LL^2}^{2}  \Big] |\bE^l |_{\LL^2}^2 
 + \frac{1}{2}  |\bQ \bE^0|_{\LL^2}^2  \nonumber \\
 &\; + C(\nu,\epsilon)   \, k\,   \Big[  |\nabla \bU^0  |_{\LL^2}^2  
 +   \max_{1\leq l\leq N}    \big[ |\bu^l |_{\LL^2}^4 + |\bU^l |_{\LL^2}^4  \big]  
+  \Big(\max_{1\leq l\leq N} \|\bu^l\|^2_V \Big)  \Big( \sum_{l=1}^N\| \bu^l - \bu^{l-1}\|^2_V \Big)\Big] 
  \nonumber\\
&+ C(\nu,\epsilon)\,  h^2 \Big( 1+ \max_{1\leq l\leq N} \|\bu^l\|_V^4 + \max_{1\leq l\leq N} |\bU^l |_{\LL^2}^4\Big)
\Big[ 1+  \Big( \nu\, k\, \sum_{l=1}^N |A\bu^l |_{\LL^2}^2\Big) \Big] .
\end{align} 
For $n=0, ..., N$ set 
\[ Y_n:= |\bE^n|_{\LL^2}^2 + \nu (1-5\epsilon) k \sum_{l=1}^n |A^{\frac{1}{2}}  \bE^l|_{\LL^2}^2,\]
Using \eqref{bE-Q0bE}, we deduce that for $\epsilon \in \big(0,\frac{1}{5}\big) $ and $\lambda >0$ we have   
\[ 0\leq Y_n \leq \tilde{Y} + k \tilde{g} \sum_{l=0}^{n-1} Y_l,\quad n=0, ..., N
\] 
where 
\begin{equation}		\label{defgt}
\tilde{g}= 2\, (1+\lambda) \Big[ 1+ \frac{\bar{C}^2}{8\epsilon \nu} \max_{1\leq j\leq N} 
 |A^{\frac{1}{2}} \bu^{l}|_{\LL^2}^{2} \Big],
\end{equation}
and
\begin{align}	\label{defYt}
\tilde{Y}=&\; (1+\lambda) |\bE^0|_{\LL^2}^2 +   C(\nu,\epsilon, \lambda)   \, k\,    |\nabla \bU^0  |_{\LL^2}^2  \nonumber \\
& + C(\nu, \epsilon, \lambda) \, k\, \Big[   \max_{1\leq l\leq N}    \big[ |\bu^l |_{\LL^2}^4 + |\bU^l |_{\LL^2}^4  \big]  
+  \Big(\max_{1\leq l\leq N} \|\bu^l\|^2_V \Big)  \Big( \sum_{l=1}^N\| \bu^l - \bu^{l-1}\|^2_V \Big)\Big] 
  \nonumber\\
&+ C(\nu,\epsilon,\lambda ) h^2 \Big( 1+ \max_{1\leq l\leq N} \|\bu^l\|_V^4 + \max_{1\leq l\leq N} |\bU^l |_{\LL^2}^4\Big)
\Big[ 1+  \Big( \nu\, k\, \sum_{l=1}^N |A\bu^l |_{\LL^2}^2\Big) \Big] .
\end{align} 
Using the discrete Gronwall lemma, we deduce that a.s.  
\[  \max_{0\leq n\leq N} Y_n \leq \tilde{Y} \, e^{\sum_{l=1}^N  k \tilde{g}} = \tilde{Y}\, e^{T  \tilde{g}}.\]
Let $p,q\in (1,+\infty)$ be conjugate exponents; H\"older's inequality implies that
\[ \EE\Big( \max_{0\leq n\leq N} Y_n\Big)  \leq \big\{\EE\big( |\tilde{Y}|^q\big) \big\}^{\frac{1}{q}} \big\{ \EE\big( e^{pT\tilde{g}}\big) \big\}^{\frac{1}{p}}. \]
 Since ${\rm Tr}\, (Q) < \frac{4\, \nu^2}{5\, \bar{C}^2\, T}$, we deduce that 
$\frac{5 \bar{C}^2 T}{4\nu} < \tilde{\alpha}_0 $. Therefore,  we may choose  $\epsilon \in \big(0,
\frac{1}{5}\big) $ (close to $\frac{1}{5}$), $\lambda >0$ (close to 0)
and $p\in (1,+\infty)$ (close to 1) such that $2 (1+\lambda) \frac{\bar{C}^2 T}{8\epsilon \nu}\, p
 < \tilde{\alpha}_0$. Hence \eqref{exp-moments-ul_al}  implies that for this choice of $p$ and $\epsilon$ we have $\EE\big( e^{p\, T\tilde{g}}\big) <\infty$.


Let $q$ be the conjugate exponent of $p$; as in the proof of Theorem \ref{E-E1}, the Young inequality together with 
\eqref{U0}, \eqref{C(T,q)}, \eqref{C(T,4)}  and  \eqref{U2}  implies  $\EE\big( |\tilde{Y}|^q\big)\leq C(k+h^2)$. This concludes the proof.
\end{proof}
Theorems \ref{Loc_cv_Euler} and \ref{E-E2} prove the following $L^2(\Omega)$ rate of convergence, which is the main result of this section for divergence-free finite elements.
\begin{theorem}		\label{strong_space-time-divfree}
Let $u_0$ and  $\bU^0$  be as in Theorem \ref{E-E2}, and suppose that  ${\rm Tr}(Q) < \frac{4\, \nu^2}{5\, \bar{C}^2 \, T} $.  Then 
given any  $\eta \in (0,1) $   we have for $N$ large enough, $k=T/N$ and $h\in (0,1)$
\begin{equation}		\label{strong_gene_2-divfree}
\EE \Big( \max_{0\leq l\leq N} |u(t_l)-\bU^l|_{\LL^2}^2 + \frac{T}{N} \sum_{l=1}^N \big|A^{\frac{1}{2}}  \big(u(t_l)) -  \bU^l\big) \big|_{\LL^2}^2 \Big)
\leq  C\big[ k^\eta + h^2 \big]. 
\end{equation}
\end{theorem} 

\bigskip

\section{Strong convergence for a random initial condition}  \label{s-u0-random}
Throughout this section, we suppose that $u_0$ is independent of the noise $W$, and that $\EE\big( e^{\gamma_0 \|u_0\|_V^2}\big) <\infty$
for some $\gamma_0>0$. We state the  strong convergence results for the time scheme $\bu^l$ and the space-time scheme $\bU^l$.
The following theorem proves the strong speed of convergence for the time scheme; its proof is a variant of that of Theorem \ref{Loc_cv_Euler}.
\begin{theorem} \label{Loc_cv_Euler_random} 
Let \eqref{K0} hold; suppose that $u_0$ is ${\mathcal F}_0$-measurable such that \eqref{gamma0} holds.
Suppose that for some $\mu \in (0,1)$ we have  ${\rm Tr}(Q) < \mu\, \frac{2  \nu^2}{\bar{C}^2\,  T}$ and 
$\gamma_0 \geq \frac{T\, \bar{C}^2}{2\,\nu\, (1-\mu)}$.
Then, given any  $\eta \in (0, 1)$, there exists a positive constant $C$ such that
\begin{equation}				\label{moments1_alea}
\EE\Big( \max_{1\leq j \leq N}    |e_j|_{\LL^2}^2   + {\nu} \frac{T}{N}
\sum_{j=1}^N   |A^{\frac{1}{2}} e_j|_{\LL^2}^2   \Big) \leq C \Big(\frac{T}{N}\Big)^\eta 
\end{equation}
for $N$ large enough. 
\end{theorem} 
\begin{proof} Let $\tilde{\alpha}_0$ be defined by \eqref{alpha0}.
The assumptions on $\gamma_0$ and ${\rm Tr}(Q)$ imply 
\[ \frac{T\bar{C}^2}{2\nu}\; \frac{\gamma_0+\tilde{\alpha}_0}{\tilde{\alpha}_0 \gamma_0} = \frac{T\bar{C}^2}{2\nu}\,\frac{  {\rm Tr}(Q)}{\nu}
+ \frac{T\bar{C}^2}{2\nu} \; \frac{1}{\gamma_0} <  \mu + (1-\mu) =  1.\]
Therefore,  $\frac{T\, \bar{C}^2}{2\nu} < \tilde{\alpha}_0\, \frac{\gamma_0}{\gamma_0 + \tilde{\alpha}_0}$, and
 we may choose $p\in (1,+\infty)$ (close to 1)
and $\delta_1\in (0,1)$ (close to 1) such that $pT\frac{ \bar{C}^2}{2\delta_1 \nu} <\tilde{\alpha}_0\, \frac{\gamma_0}{\gamma_0 + \tilde{\alpha}_0}$. 
The integrability property \eqref{exp-moments-udet} 
implies 
\[ 
\EE \Big[ \exp\Big( pT  \frac{\bar{C}^2}{2\delta_1 \nu} \sup_{s\in [0,T] } |A^{\frac{1}{2}} u(s)|_{\LL^2}^2  \Big)\Big] =C <\infty. 
 \] 
 Using the upper estimate \eqref{maj_E_sup_ej} and the inequality \eqref{maj_Z1_q} -- \eqref{maj_Z2_q}, 
we then conclude the proof as that of Theorem \ref{Loc_cv_Euler}.
\end{proof}
\smallskip

The next result gives the strong speed of convergence of the space-time discretization for a random  initial condition. It requires some lower bound on the
exponential moment of $\|u_0\|_V^2$ and some related upper bound on the strength of the noise. 
\begin{theorem}	\label{cor_FEM_gene_alea}
Let \eqref{K0} hold and let  $u_0$  be ${\mathcal F}_0$-measurable
such that  \eqref{gamma0} is satisfied  for some $\gamma_0>0$. Let $\bU^0$ be ${\mathcal F}_0$-measurable, taking values in 
$ \HH_h$, and such that  $\EE\big( \|\bU^0\|_{{\mathbb W}^{1,2}}^{q_0}\big)<\infty$ for some some $q_0$ large enough.
 Suppose furthermore that \eqref{U0} is satisfied for this value of $q_0$. 
 Let $\{(\bU^l, \Pi^l)\}_l$ be solution of  Algorithm 1.  
 Suppose that  for some $\mu \in (0,1)$, 
 $${\rm Tr}(Q) < \mu \frac{4\, \nu^3}{13  \, \big( T\, \nu\, \bar{C}^2    + 4 \, \sigma^2 \big)} $$  
 and
$$\gamma_0 \geq \frac{13 \, ( T\, \nu\, \bar{C}^2   + 2 \sigma^2) }{4\, \nu^2\, (1-\mu)}.$$
  
\noindent Then,  for  any $\eta\in (0,1)$, there exists a positive constant $C$ such that   
\begin{equation}		\label{strong_gene_2_alea}
\EE \Big( \max_{0\leq l\leq N} |u(t_l)-\bU^l|_{\LL^2}^2 + \frac{T}{N} \sum_{l=1}^N \big|A^{\frac{1}{2}} (u(t_l)) - A^{\frac{1}{2}} \bU^l\big|_{\LL^2}^2 \Big)
\leq  C\big[ k^\eta + h^2 \big]. 
\end{equation}
 for $N$ large enough and  $h\in (0,1)$.
\end{theorem} 
\begin{proof} As in Section \ref{s4}, we at first prove that under the hypotheses of this theorem, we have for $\bE^n=\bu^n-\bU^n$,
\begin{equation}		\label{strong_gene_alea}
\EE\Big( \max_{0\leq n\leq N}  |\bE^n|_{\LL^2}^2 + k \sum_{l=1}^N |A^{\frac{1}{2}} \bE^l|_{\LL^2}^2 \Big) \leq 
C  \big[ k+ h^2 \big]. 
\end{equation} 
The upper estimate \eqref{strong_gene_2_alea} is a consequence of \eqref{strong_gene_alea} and Theorem \ref{Loc_cv_Euler}.

The proof is similar to that of Theorem \ref{E-E1}, and we start with \eqref{upper_Xn}, where $\sum_{l=0}^{N-1} g_l$ is defined in \eqref{sum-gl}
for $\epsilon \in (0, \frac{1}{13})$.

The hypotheses on $\gamma_0$ and ${\rm Tr}(Q)$ imply for $\tilde{\alpha}_0=\frac{\nu}{{\rm Tr}(Q)}$,
 $\tilde{\beta}_0=\frac{\tilde{\alpha}_0
\gamma_0}{\gamma_0+\tilde{\alpha}_0}$ and $\tilde{\beta}_1=\frac{\tilde{\alpha}_0 \gamma_0}{2\gamma_0+\tilde{\alpha}_0}$, 
\begin{align*}
\frac{13}{4\, \nu^2 } \Big(  \frac{\bar{C}^2\,T\, \nu}{\tilde{\beta}_0} &+ \frac{2\sigma^2}{\tilde{\beta}_1}\Big) = 
\frac{13}{4\, \nu^2 }  \Big( \frac{\bar{C}^2\, T\, \nu (\tilde{\alpha}_0+\gamma_0)}{\gamma_0\, \tilde{\alpha}_0} 
+ \frac{2\, \sigma^2 \big( 2\gamma_0+\tilde{\alpha}_0)}{\gamma_0\tilde{\alpha}_0}\Big) \\
=& 
\frac{13}{4\, \nu^2 } \times  \frac{\bar{C}^2\, T\, \nu +2\, \sigma^2 }{\gamma_0} + \frac{13}{4\, \nu^2 } \times  \frac{(\bar{C}^2 T\, \nu+4\sigma^2)\,  {\rm Tr}(Q)}{\nu}
<  (1-\mu) + \mu =1.
\end{align*}
Let $\lambda \in (0,1)$ (close to 0) and $\epsilon \in (0, \frac{1}{13})$ (close to $\frac{1}{13}$), such that 
for  $p_1=\frac{4 \epsilon \nu \tilde{\beta_0}}{\bar{C}^2 T (1+\lambda)^2}$  and $p_2=\frac{2\epsilon \nu^2 \tilde{\beta}_1}{\sigma^2 (1+\lambda)^2}$,
we have
\[ \frac{1}{p}:= \frac{1}{p_1} + \frac{1}{p_2} = \frac{(1+\lambda)^2}{\epsilon\, 4\, \nu^2}\,  \Big( \frac{\bar{C}^2\, T\, \nu}{\tilde{\beta}_0} + \frac{2\, \sigma^2}{\tilde{\beta}_1}
\Big)  <1.\]  
 Furthermore, 
  $p_1 \frac{(1+\lambda) \bar{C}^2 T}{4\epsilon \nu}=\frac{\tilde{\beta}_0}{1+\lambda}<\tilde{\beta}_0$ and 
  $p_2 \frac{(1+\lambda) \sigma^2}{2 \epsilon \nu^2 }= \frac{\tilde{\beta}_1}  {1+\lambda} <\tilde{\beta}_1$. Therefore,  Theorem \ref{exp-mom} (ii),
  \eqref{exp-moments-ul_al} and \eqref{exp-moments-Aul_al}
  for a random initial condition $u_0$ and the H\"older inequality yield $\EE\big( p \sum_{l=0}^{N-1} g_l\big) <\infty$. 
  
 Let $q=\frac{p}{p-1}$ denote the conjugate exponent of $p$; the end of the proof of Theorem \ref{E-E1} concludes the proof of \eqref{strong_gene_alea},
 and hence that of the Theorem. 
 \end{proof} 
\smallskip

We finally give an analog of Corollary \ref{strong_space-time-divfree} for divergence-free finite elements when the initial condition $u_0$ is random.

\begin{theorem}
Suppose that the finite elements are divergence free, that is ${\bf V}_h \subset V$. Let  \eqref{K0} hold,  and $u_0$ be ${\mathcal F}_0$-measurable and
satisfying  \eqref{gamma0}
for some positive constant $\gamma_0$.   Consider the solution $\{\bU^l\}_l$ of Algorithm 2, 
 and suppose that for some constant $\mu \in (0,1)$ we have  ${\rm Tr}(Q) < \mu\, \frac{4\, \nu^2}{5\, \bar{C}^2\, T}$ and 
$\gamma_0\geq \frac{5 \, \bar{C}^2\,  T}{4\, \nu\,  (1-\mu)} $.  
  Then for any $\eta\in (0,1)$ we have for $N$ large enough, $k=T/N$ and $h\in (0,1)$
\begin{equation}
\EE \Big( \max_{0\leq l\leq N} |u(t_l)-\bU^l|_{\LL^2}^2 + \frac{T}{N} \sum_{l=1}^N \big|A^{\frac{1}{2}}  \big(u(t_l)) -  \bU^l\big) \big|_{\LL^2}^2 \Big)
\leq  C\big[ k^\eta + h^2 \big]. 
\end{equation}
\end{theorem}
\begin{proof}
Using Theorem \ref{Loc_cv_Euler_random}, it is sufficient to prove that for large $N$ and $h\in (0,1)$, 
\[\EE\Big( \max_{0\leq n\leq N}  |\bE^n|_{\LL^2}^2 +  k \sum_{l=1}^N |A^{\frac{1}{2}}  \bE^l|_{\LL^2}^2 \Big) \leq 
  C\,  \big( k+ h^2 \big). 
\]
The proof is similar to that of Theorem \ref{E-E2}.   The hypotheses on $\gamma_0$ and ${\rm Tr}(Q)$ imply  
\[ \frac{5\, T\, \bar{C}^2}{4\nu}\; \frac{\tilde{\alpha}_0+\gamma_0}{\tilde{\alpha}_0\gamma_0} = \frac{5\, T\, \bar{C}^2} {4\nu \gamma_0} +
\frac{5\, T\,\bar{C}^2\,    {\rm Tr}(Q)}{4\nu^2} < (1-\mu) + \mu =1.
\]
Therefore, we may choose $\lambda \in (0,1)$ (close to 0), $\epsilon \in (0, \frac{1}{5})$ (close to $\frac{1}{5}$) and $p\in (1,\infty)$ (close to 1) such that
$p\, (1+\lambda)\, \frac{\bar{C}^2\, T}{4\, \epsilon\, \nu} < \frac{\tilde{\alpha}_0 \gamma_0}{\gamma_0+\tilde{\alpha}_0}$. For this choice of $p$, using
\eqref{exp-moments-ul_al} in Theorem \ref{exp-mom} for a random initial condition, we deduce that $\EE\big( \exp(pTg)\big)<\infty$. 
We then complete the proof as in the argument used
at the end of the proof of Theorem \ref{E-E2} with the conjugate exponent $q=\frac{p}{p-1}$. 
\end{proof}

\section{Proofs of the time-regularity  of $u$} \label{s-proofs-time} In this section, we give the proofs of Lemmas \ref{Holder-L2} and \ref{holder-q}. They
rely on semi-groups arguments. 

Let $\{S(t)\}_{ t\geq 0}$ be the semi-group generated by $-\nu A$, that is $S(t)=e^{-\nu t A}$. The two first upper estimates of the following lemma are classical 
 (see e.g. \cite{CarPro}, Lemma 2.2). The last one describes the link between the Stokes operator and the bilinear term B (see Sobolevski or Giga-Miyakama
 \cite{GiMi}). 

 \begin{lemma} 		\label{semi-group}
 (i)  For every $b>0$,  there exist  positive constants $C(b) $ and $\tilde{C}(b)$ such that for every $t>0$ and $k=0,1$ 
 \begin{align}
  \| A^b\, S(t) \|_{{\mathcal L}( \LL^2, \LL^2)}
  & \leq  C(b)\, t^{-b},		\label{AS}\\
  \| A^{-b}\, \big( {\rm Id}\, - S(t)\big)  \|_{{\mathcal L}( \LL^2, \LL^2)} 
  & \leq  \tilde{C}(b)\, t^{b}.		\label{A(I-S)}
  \end{align} 
    (ii) For $u,v\in V$ we have  for any $\delta \in \big(0,1\big)$ 
  \begin{align}		\label{A.B}
  \big| A^{-\delta} B(u,v) \big|_{\LL^2} \leq C \big| A^{\frac{1}{2}} u\big|_{\LL^2}\, \big| A^{\frac{1}{2}} v\big|_{\LL^2}. 
  \end{align} 
 \end{lemma}
 We at first prove Lemma \ref{Holder-L2} about moments of the $\LL^2$-norms of time increments of the solution $u$ to \eqref{2D-NS}. 
 
 \noindent{\it Proof of Lemma \ref{Holder-L2}} We project \eqref{2D-NS} on divergence-free fields. 
 For $0\leq s < t \leq T$ we have $u(t)-u(s) = \sum_{i=1}^3 T_i(s,t)$, where using the mild formulation of the strong solution to \eqref{2D-NS} we have
 \begin{align*}
 T_1(s,t):=&\;  S(t) u_0 - S(s) u_0,\\
 T_2(s,t):=&\,  - \int_0^t S(t-r) B\big(u(r),u(r)\big)\,  dr + \int_0^s S(s-r) B\big(u(r),u(r)\big)\,  dr ,\\
 T_3(s,t):=& \, \int_0^t S(t-r) \,   d  {W}(r) - \int_0^s S(s-r)\,  d {W}(r).
 \end{align*}
 Using \eqref{A(I-S)} with $k=0$ and $\sup_{s\in [0,T]} \| S(s)\|_{{\mathcal L}(\LL^2,\LL^2)} <\infty$ we deduce
 \begin{equation}		\label{upper_T1}
 |T_1(s,t)|_{\LL^2} = \big| S(s) \, A^{-\frac{1}{2}} \, \big[ S(t-s)-{\rm Id}\, \big] \, A^{\frac{1}{2}} u_0\big|_{\LL^2} \leq C\, |t-s|^{\frac{1}{2}} \, \|u_0\|_V.
 \end{equation}
 We split $T_2(s,t)= T_{2,1}(s,t)+T_{2,2}(s,t)$, where
 \begin{align*}
 T_{2,1}(s,t)=&\,  - \int_0^s S(s-r) \big[ S(t-s) - {\rm Id}\, \big] B\big(u(r),u(r)\big)\, dr,\\
   T_{2,2}(s,t)=&\, -\int_s^t S(t-r)\, B\big(u(r),u(r)\big) dr. 
 \end{align*}
 The Minkowski inequality, \eqref{AS} and  \eqref{A(I-S)} for $k=0$,  and \eqref{A.B}   for $\delta = \frac{1}{8}$ imply for any 
  $\epsilon >0$
 \begin{align}		\label{upper_T21}
 |T_{2,1}(s,t)|_{\LL^2} \leq & \, \int_0^s \big| A^{1-\epsilon}  S(s-r)\; A^{-(\frac{7}{8}-\epsilon)} \big [S(t-s)- {\rm Id} \big]\; A^{-\frac{1}{8}} B\big( u(r),u(r)\big)\big|_{\LL^2}\, 
 dr \nonumber \\
 \leq & \,  C\, (t-s)^{\frac{7}{8}-\epsilon} \, \sup_{r\in [0,s]}\|u(r)\|_V^2 \, \int_0^s (s-r)^{-1+\epsilon} dr \nonumber \\
  \leq&\,  C (t-s)^{\frac{7}{8}-\epsilon} \sup_{r\in [0,s]}\|u(r)\|_V^2 . 
 \end{align}
 A similar argument based on the Minkowski inequality,  \eqref{AS} for $k=0$ and \eqref{A.B} implies
  \begin{align}		\label{upper_T22}
 |T_{2,2}(s,t)|_{\LL^2} \leq & \, \int_s^t \big| A^{\frac{1}{8}} S(t-r)\; A^{-\frac{1}{8}} B\big( u(r),u(r)\big)\big|_{\LL^2}  dr \nonumber \\
 \leq&\,  C \sup_{r\in [s,t]} \|u(r)\|_V^2 \int_s^t (t-r)^{-\frac{1}{8}} 
 dr 
\leq \, C\, (t-s)^{\frac{7}{8}} \, \sup_{r\in [s,t]}\|u(r)\|_V^2 .
 \end{align}		
Since  $\EE(\|u_0\|_V^{4q})<\infty$, using \eqref{bound_u} we have $\EE\big(\sup_{r\in [0,T]} \|u(r)\|_V^{4q}\big) <\infty$. Therefore, 
the upper estimates \eqref{upper_T1}--\eqref{upper_T22} imply  for $\epsilon := \frac{3}{8}$ 
 \[ \EE\big( |u(t)-u(s)|_{\LL^2}^{2q}\big) \leq  C(q)\Big[ |t-s|^q  \big[1+ \EE\big(\|u_0\|_V^{4q}\big) \big] 
 + \EE\big(|T_3(s,t)|_{\LL^2}^{2q}\big) \Big]. 
 \]
 We split $T_3(s,t)=T_{3,1}(s,t)+T_{3,2}(s,t)$, where
  \begin{align*}
 T_{3,1}(s,t)=\   \int_0^s S(s-r) \big[ S(t-s) - {\rm Id}\, \big] \,  d{W}(r), \quad 
   T_{3,2}(s,t)= \int_s^t S(t-r)\,   d{W}(r). 
 \end{align*}
The Burkholder-Davies-Gundy inequality, 
\eqref{AS} and \eqref{A(I-S)}  for $k=0$ imply that for $q\in [2,\infty)$
and $\tilde{\epsilon}\in (0, \frac{1}{2})$
\begin{align}		\label{upper_T31}
\EE\big( |T_{3,1}|_{\LL^2}^{2q}\big)\leq &
\EE\Big( \Big| \int_0^s  A^{\frac{1}{2}-\tilde{\epsilon}} S(s-r)\, \big[ A^{-\frac{1}{2}}(S(t-s)-{\rm Id})\big]
A^{\tilde{\epsilon}}\,   d{W}(r)\Big|^{2q} \Big)
 \nonumber \\
\leq & \, C_q \; \EE\Big( \Big| \int_0^s \| A^{\frac{1}{2}-\tilde{\epsilon}}S(s-r)\, \big[ A^{-\frac{1}{2}}
(S(t-s)-{\rm Id})\, \big]\|_{{\mathcal  L}(\LL^2,\LL^2)}^2\,
{\rm Tr}( A^{\tilde{\epsilon}} Q)\; dr\Big|^q\Big)
  \nonumber \\
\leq &\, C_q \; K_0^q\,   \Big| \int_0^s \| A^{\frac{1}{2}-\tilde{\epsilon}} S(s-r)\|_{{\mathcal L}( \LL^2,\LL^2)}^2
\| A^{-\frac{1}{2}} \big[ S(t-s)-{\rm Id}\,\big] \|_{{\mathcal L}(\LL^2,\LL^2)}^2 dr \Big|^q \nonumber \\
\leq &\, C(q,\tilde{\epsilon})  \, K_0^q \, |t-s|^{q} \, \Big( \int_0^s (s-r)^{-1+2\tilde{\epsilon}} dr\big)^q\nonumber
\\
\leq &\, C(q,\tilde{\epsilon},  K_0 )  |t-s|^{q}. 
\end{align}
 A similar argument, using $\sup_{r\in [0,T]} \| S(r)\|_{{\mathcal L}(\LL^2,\LL^2)} <\infty$, implies
 \begin{align}		\label{upper_T32}
 \EE\big( |T_{3,2}|_{\LL^2}^{2q}\big) \leq & \, C_q \; \EE\Big( \Big| \int_s^t \| S(t-r) \|_{{\mathcal L}(\LL^2,\LL^2)}^2\, 
{\rm Tr}(Q) \; dr\Big|^q\Big)
  \nonumber \\
\leq &\, C_q \, {\rm Tr}(Q)^q  \, \Big| \int_s^t \| S(t-r)\|_{{\mathcal L}(\LL^2,\LL^2)}^2\, dr\Big|^q \leq 
C(q,  {\rm Tr}(Q))  |t-s|^q.
 \end{align}
  Given $\eta \in (0,1]$, 
 the  upper estimates \eqref{upper_T1}--\eqref{upper_T32} conclude the proof of \eqref{holder-L2}.  
 \hfill $\square$
\bigskip

We next prove Lemma \ref{holder-q} about moment estimates of the $V$-norm of time increments. In order to have a constants which do not
depend on the time mesh, we have to consider time integrals of these increments.
 
 \noindent {\it Proof of Lemma \ref{holder-q}}
  We at first prove \eqref{holder-V_q}. Parts on the argument are similar to that in the proof of Lemma \ref{Holder-L2}; 
 we include them for the sake of completeness. 
 For $s,t\in [0,T]$ with $s<t$ we use the decomposition of $u(t)-u(s)$ introduced in the proof of Lemma \ref{Holder-L2}. 
 
 The upper estimates for the
integrals of $\LL^2$-norms  of  the time increments of $T_1$ and $T_2$ on the intervals $[t_j, t_{j+1})$, $j=1, ...,N$ are straightforward consequences 
 of \eqref{upper_T1} and \eqref{upper_T21}--\eqref{upper_T22} respectively; they require $\eta\in (0, 1)$.  We next give upper estimates of the gradient
 of the time increments in these deterministic  integrals.  Set
  \begin{align*}
 \tilde{T}_1(s,t):=&\;  A^{\frac{1}{2}} S(t) u_0 -  A^{\frac{1}{2}} S(s) u_0,\\
 \tilde{ T}_{2,1}(s,t)=&\,  - \int_0^s A^{\frac{1}{2}}   S(s-r) \big[ S(t-s) - {\rm Id}\, \big] B\big(u(r),u(r)\big) \, dr,\\
  \tilde{T}_{2,2}(s,t)=&\, -\int_s^t A^{\frac{1}{2}} S(t-r)\, B\big(u(r),u(r)\big)  dr. 
\end{align*} 
 Using once more \eqref{AS} and \eqref{A(I-S)} with $k=0$, we deduce that for any $\alpha >0$, 
 \begin{equation} 	\label{upper_tildeT1}
 |\tilde{T}_1(s,t)|_{\LL^2} = \big| A^\alpha S(s) \, A^{-\alpha}\big[ S(t-s)-{\rm Id}\,\big] \, A^{\frac{1}{2}}u_0 \big|_{\LL^2} \leq C\, s^{-\alpha}\, |t-s|^{\alpha} \, \|u_0\|_V.
 \end{equation}
 The Minkowski inequality,  the upper estimates \eqref{AS}--\eqref{A(I-S)} with $k=0$  and \eqref{A.B}  with $\delta \in \big( 0,\frac{1}{2}\big)$
 yield for $\beta \in \big(0, \frac{1}{2}-\delta\big)$
 \begin{align}		\label{upper_tildeT21}
 \big|\tilde{T}_{2,1}(s,t)\big|_{\LL^2} \leq & \int_0^s \big| A^{\delta + \frac{1}{2} + \beta}S(s-r)\; A^{-\beta}[S(t-s)-{\rm Id}\,] \; 
  A^{-\delta}B\big(u(r),u(r)\big) \big|_{\LL^2}\, dr \nonumber \\
& \leq   \, C  \Big( \!\int_0^s \! (s-r)^{-(\frac{1}{2}+\beta+\delta)} dr \Big) \sup_{r\in [0,s]} \|u(r)\|_V^2\, |t-s|^\beta  \nonumber \\
&\leq  C(\beta, \delta, T)  \sup_{r\in [0,s]} \|u(r)\|_V^2 |t-s|^\beta .
 \end{align}
 Using once more the Minkowski inequality, \eqref{AS}  for $k=0$ and \eqref{A.B} with $\delta \in \big( 0, \frac{1}{2}\big)$, we obtain
 \begin{align}		\label{upper_tildeT22}
 \big|\tilde{T}_{2,2}&(s,t)\big|_{\LL^2} \leq  \int_s^t  \big| A^{\delta + \frac{1}{2} }S(t-r)\; 
  A^{-\delta}B\big(u(r),u(r)\big) \big|_{\LL^2}\, dr \nonumber \\
 \leq & \,C \Big( \int_s^t (t-r)^{-\frac{1}{2}-\delta}\, dr \Big) \, \sup_{r\in [s,t]} \|u(r)\|_V^2 = C(T, K_0) \, \sup_{r\in [s,t]} \|u(r)\|_V^2\, |t-s|^{\frac{1}{2}-\delta} .
 \end{align}
 The upper estimates \eqref{upper_tildeT1}--\eqref{upper_tildeT22} imply for $\tilde{T}_2:= \tilde{T}_{2,1}+\tilde{T}_{2,2}$, 
 $\alpha, \delta \in \big( 0, \frac{1}{2}\big) $ and $\beta \in \big( 0, \frac{1}{2}-\delta \big)$
 \begin{align*}		
 \EE\Big( \Big|& \sum_{j=1}^N \int_{t_{j-1}}^{t_j}  \big[\big| \tilde{T}_1(s,t_j)\big|_{\LL^2}^2  + \big| \tilde{T}_{2}(s,t_j)\big|_{\LL^2}^2 \big] ds \Big|^q \Big)
  \leq  C(q, T,   \alpha,\beta) \\ 
\quad \times &  \Big\{ \EE\Big( \Big| \sum_{j=1}^N \int_{t_{j-1}}^{t_j} \! \!\!s^{-2\alpha} |t_j-s|^{2\alpha} \,\|u_0\|_V^2 \,ds\Big|^q\Big) 
  + \EE\Big( \Big| \sum_{j=1}^N \int_{t_{j-1}}^{t_j} \! \! |t_j-s|^{2\beta} \, \sup_{r\in [0,T]} \|u(r)\|_V^4 \, ds\Big|^q \Big) \Big\} \\
 \leq & \; C(q, T,\alpha,\beta) \Big\{ \Big( \frac{T}{N}\Big)^{2\alpha q} \Big| \int_0^T \!\! s^{-2\alpha} ds\Big|^q \EE(\|u_0\|_V^{2q}) 
 + N^q\, \Big( \frac{T}{N}\Big)^{(1+2\beta)q} \EE\Big( \sup_{r\in [0,T]} \|u(r)\|_V^{4q}\Big) \Big\} \\ 
 \leq & \; C\big(q,T,\alpha,\beta, K_0\big)  \Big(\frac{T}{N}\Big)^{2(\alpha \wedge \beta)q}\, \big[1+ \EE\big(\|u_0\|_V^{4q}\big) \big], 
 \end{align*} 
 where the last upper estimate is a consequence of \eqref{bound_u}. 
Hence  for  $\alpha \in \big( 0, \frac{1}{2}\big) $ and  $0<\beta < \frac{1}{2}-\delta$ for some $\delta \in \big( 0, \frac{1}{2}\big)$, we obtain 
\begin{equation} \label{upper_middle_V}
  \EE\Big( \Big| \sum_{j=1}^N \int_{t_{j-1}}^{t_j}  \big[\big|A^{\frac{1}{2}} {T}_1(s,t_j)\big|_{\LL^2}^2  + 
  \big| A^{\frac{1}{2}}  {T}_{2}(s,t_j)\big|_{\LL^2}^2 \big] ds \Big|^q \Big) \leq 
  C   \Big(\frac{T}{N}\Big)^{2(\alpha \wedge \beta)q}
 \end{equation}
 some some positive constant $C$. 

 We finally deal with the  stochastic integrals and write $T_3(s,t)=T_{3,1}(s,t) + T_{3,2}(s,t)$ as in the proof of Lemma \ref{Holder-L2}.
 Using the H\"older and Burkholder-Davies-Gundy inequality together with the upper estimates \eqref{AS}--\eqref{A(I-S)} with $k=1$ and \eqref{A.B}, we deduce 
 for  $\epsilon \in \big(0,\frac{1}{2}\big) $
 \begin{align}		\label{upper_tildeT31}
 \EE\Big( \Big| \sum_{j=1}^N& \int_{t_{j-1}}^{t_j}\!\! |A^{\frac{1}{2}}{T}_{3,1}(s,t_j)|_{\LL^2}^2\, ds \Big|^q\Big) \leq N^{q-1} \sum_{j=1}^N 
 \EE\Big( \Big| \int_{t_j}^{t_j} \!\! |A^{\frac{1}{2}} {T}_{3,1}(s,t_j)|_{\LL^2}^2\, ds\Big|^q\Big) \nonumber \\
 \leq & \, C N^{q-1} \sum_{j=1}^N \Big(\frac{T}{N}\Big)^{q-1} \!\! \int_{t_{j-1}}^{t_j} \!\!  \EE\Big(  \Big| \int_0^s A^{\frac{1}{2}}
   S(s-r) \big[ S(t_j-s)-{\rm Id}\, \big]  d{W}(r) \Big|_{\LL^2}^{2q} \Big) \, ds \nonumber \\
 \leq & \, C(q,T) \sum_{j=1}^N \int_{t_{j-1}}^{t_j} \! \EE\Big( \Big| \int_0^s \! \big\| 
   S(s-r) \big[ S(t_j-s)-{\rm Id}\,\big] \big\|_{ {\mathcal L}(\LL^2,\LL^2)}^2
 K_0  \, dr \Big|^q \Big) ds \nonumber \\
 \leq & \, C(q,T)   K_0^q  \, \sum_{j=1}^N \int_{t_{j-1}}^{t_j} \! \EE\Big( \Big| \int_0^s \! \!
  \big\| A^{\frac{1}{2}-\epsilon} S(s-r)\big\|_{ {\mathcal L}(\LL^2,\LL^2)}^2
 \nonumber \\
 &\qquad \qquad \times 
 \,\big \| A^{-\frac{1}{2}+\epsilon}\big[ S(t_j-s) - {\rm Id}\, \big]\big\|_{{\mathcal L}(\LL^2,\LL^2)}^2 \, 
  \, dr \Big|^q\Big)\, ds
 \nonumber \\
 \leq &\, C\big(q,T,K_0\big) \, \Big( \frac{T}{N}\Big)^{(1-2\epsilon)q}  \; \sum_{j=1}^N \int_{t_{j-1}}^{t_j} \Big| \int_0^s (s-r)^{-1+2\epsilon} dr
 \Big|^q \, ds \nonumber \\
 \leq & \, C(q,T, K_0 ,\epsilon)  \Big( \frac{T}{N}\Big)^{(1-2\epsilon)q}.
 \end{align} 
 
 A similar argument implies
 \begin{align}		\label{upper_tildeT32}
 \EE\Big( \Big| \sum_{j=1}^N &\int_{t_{j-1}}^{t_j}\!\!  |A^{\frac{1}{2}} {T}_{3,2}(s,t_j)|^2_{\LL^2} \, ds \Big|^q\Big) \nonumber \\
 \leq&\,  N^{q-1} \sum_{j=1}^N 
 \Big(\frac{T}{N}\Big)^{q-1} \int_{t_{j-1}}^{t_j} \EE \Big(\Big|  \int_s^{t_j}  A^{\frac{1}{2}} S(t_j-s) d{W}(r) \Big|_{\LL^2}^{2q} \Big) ds \nonumber \\
 \leq &\, C\big(q,T,K_0\big)\,  \sum_{j=1}^N \int_{t_{j-1}}^{t_j} \EE\Big( \Big| \int_s^{t_j} \|S(t_j-s)\|_{{\mathcal L}(\LL^2,\LL^2)}^2
  dr\Big|^q\Big)ds \nonumber \\
 \leq & \, C\big(q,T, K_0\big)\,  \sum_{j=1}^N \int_{t_{j-1}}^{t_j} (t_j-s)^q ds \leq 
 C\big(q,T, K_0 \big) \Big( \frac{T}{N}\Big)^q. 
 \end{align}
 Given  ${\eta} \in ( 0, 1)$, choose $\alpha \in \big(\frac{ {\eta}}{2},\frac{1}{2}\big)$, $0<\delta, \epsilon <\frac{1}{2}(1-{\eta})$,
and $\frac{\tilde{\eta}}{2}\leq \beta < \frac{1}{2}-\delta$. The inequalities \eqref{upper_middle_V}--\eqref{upper_tildeT32} yield 
$\EE\big( \big| \sum_{j=1}^N \! \int_{t_{j-1}}^{t_j} \!\!  |A^{\frac{1}{2}} [ u(t_j)-u(s) ] |^2_{\LL^2}  ds \big|^q\big) \leq 
 {C}  \big( \frac{T}{N}\big)^{{\eta} q}$.  Using  Lemma \ref{Holder-L2}, we deduce 
 \eqref{holder-V_q} for the time increments $u(t_j)-u(s)$. The above arguments clearly show that a similar upper estimate holds for the time increments 
 $u(s)-u(t_{j-1})$;
 this completes the proof of \eqref{holder-V_q}.  
 
 Using the Gagliardo-Nirenberg inequality \eqref{interpol}, 
 the upper estimates \eqref{holder-V_q} and \eqref{holder-L2}
 and H\"older's inequality, we  obtain \eqref{holder-L4_q}. This concludes the proof. 
 \hfill $\square$

\section{Proof of   exponential moments for the time  scheme $\bu^l$} \label{s-exp-mom-Euler}
 In this section, we prove that   the time Euler scheme $\bu^l$ has exponential moments; the exponential
coefficient is related to the viscosity $\nu$ and the strength of the noise ${\rm Tr}(Q)$ (as well as to the exponent $\gamma_0$ of the
exponential moments of $|A^{\frac{1}{2}} u_0|_{\LL^2}^2$ if $u_0$ is random).
\medskip

\noindent{\it Proof of Theorem \ref{exp-mom}.} 
 We prove below \eqref{exp-moments-Aul_al}, and further indicate the changes to be done to obtain \eqref{exp-moments-ul_al}.
Note that the arguments used to prove \eqref{exp-moments-ul_al} can be found in \cite{BeMi-FE},  Theorems~8.1 and 8.3. 
\smallskip

{\bf Proof of \eqref{exp-moments-Aul_al}} 
Since Lemma \ref{moments_uN} implies that $\EE\big( \frac{T}{N}\sum_{l=1}^N |A\bu^l|_{\LL^2}^2\big) <\infty$, using integration by parts
and the fact that $A^{\frac{1}{2}} \bu^l \in V$ a.s. by \eqref{C(T,4)}, we may write  \eqref{full-imp1Bis} with $\phi = A\bu^l$;  this yields a.s. 
\begin{align*}  
 \big( A^{\frac{1}{2}} \bu^l - A^{\frac{1}{2}} \bu^{l-1}, \, A^{\frac{1}{2}} \bu^l\big) &+ \frac{T}{N}\Big[ \nu (A \bu^l,  A\bu^l) + 
  \langle B(\bu^l, \ \bu^l),\, A\bu^l\rangle \Big] 
= \big( A^{\frac{1}{2}} \Delta_lW, \, A^{\frac{1}{2}} \bu^l\big).
\end{align*} 
Using \eqref{A-B}  
we deduce
\[ \big( A^{\frac{1}{2}} \bu^l - A^{\frac{1}{2}} \bu^{l-1}, \, A^{\frac{1}{2}} \bu^l\big) + \frac{T}{N}\nu  |A\bu^l|_{\LL^2}^2
= \big( A^{\frac{1}{2}}  \Delta_l W\, , \, A^{\frac{1}{2}} \bu^l\big).\]
The identity $(a-b,a)=\frac{1}{2} \big( |a|_{\LL^2}^2 - |b|_{\LL^2}^2 + |a-b|_{\LL^2}^2\big)$ implies a.s. 
\begin{align} 
  |A^{\frac{1}{2}}   \bu^l|_{\LL^2}^2 &- |A^{\frac{1}{2}} \bu^{l-1}|_{\LL^2}^2 
+ |A^{\frac{1}{2}} ( \bu^l-\bu^{l-1})|_{\LL^2}^2   +   \nu \frac{T}{N} |A\bu^l]_{\LL^2}^2  
\nonumber \\
 =&\; 2\big( A^{\frac{1}{2}} \Delta_l W\, , \, A^{\frac{1}{2}} \bu^{l}\big)  -  \,   \frac{T}{N}\,  \nu \,  |A\bu^l]_{\LL^2}^2  \label{Gl}\\
=&\;  2 \big( A^{\frac{1}{2}}\Delta_l W\, , \, A^{\frac{1}{2}} \bu^{l-1}\big) -    \frac{T}{N}  \nu |A\bu^l]_{\LL^2}^2  +
2 \big( A^{\frac{1}{2}} \Delta_l W\, , \, A^{\frac{1}{2}} [ \bu^l - \bu^{l-1}]\big). \label{Gl-1}
\end{align} 
For $l=1$, 
the Cauchy-Schwarz and Young inequalities   
imply for $\lambda \in (0,1)$
\[ 2 \big| \big( A^{\frac{1}{2}} \Delta_l W\, , \, A^{\frac{1}{2}} \bu^{1}\big)\big| \leq \frac{1}{\lambda}
 |A^{\frac{1}{2}} \Delta_1 W|_{\LL^2}^2
+ \lambda  |A^{\frac{1}{2}}  \bu^1|_{\LL^2}^2.\]
Hence, since $1\leq \frac{1}{1-\lambda}$,  \eqref{Gl} implies 
\[ 
|A^{\frac{1}{2}} \bu^1|_{\LL^2}^2  + \frac{T}{N} \nu |A\bu^1|_{\LL^2}^2  \leq \frac{1}{1-\lambda}  |A^{\frac{1}{2}} u_0|_{\LL^2}^2 
+  \frac{1}{\lambda (1-\lambda)}
\,  |A^{\frac{1}{2}}\Delta_1 W|_{\LL^2}^2   -   \frac{T}{N} \nu |A\bu^1|_{\LL^2}^2  .
\] 
For $l\geq 2$, a similar argument  using the Cauchy-Schwarz and Young inequalities implies
\[ 
2 \big( A^{\frac{1}{2}}  \Delta_lW, \, A^{\frac{1}{2}}[\bu^l-\bu^{l-1}] \big)  \leq  2   | A^{\frac{1}{2}} \Delta_lW ) |_{\LL^2}
|A^{\frac{1}{2}} (\bu^l-\bu^{l-1})|_{\LL^2} 
\leq  | A^{\frac{1}{2}} \Delta_l W|_{\LL^2}^2 + | A^{\frac{1}{2}} (\bu^l-\bu^{l-1})|_{\LL^2}^2. 
\]
Therefore,  for $l=2, ..., N$ \eqref{Gl-1} 
implies
\[   |A^{\frac{1}{2}} \bu^l|_{\LL^2}^2 - |A^{\frac{1}{2}} \bu^{l-1}|_{\LL^2}^2    + \frac{T}{N} \nu |A\bu^ l|_{\LL^2}^2 
\leq 2\big( A^{\frac{1}{2}} \Delta_l W\, ,\,  A^{\frac{1}{2}} \bu^{l-1}\big)
 -  \frac{T}{N} \nu |A\bu^ l|_{\LL^2}^2  +  |  A^{\frac{1}{2}} \Delta_l W \|_{\LL^2}^2.\]
Adding these inequalities for $l=1, \cdots, n$, we deduce for $n = 2, \cdots, N$   
\begin{align}		\label{max-nabla-ul}
  |A^{\frac{1}{2}} \bu^n|_{\LL^2}^2   & + \nu   \frac{T}{N} 
\sum_{l=1}^n  |A\bu^l|_{\LL^2}^2  
 \leq \frac{1}{1-\lambda} |A^{\frac{1}{2}} u_0|_{\LL^2}^2  
+ \frac{1}{\lambda (1-\lambda)}  \sum_{l=1}^n |A^{\frac{1}{2}} \Delta_l W|_{\LL^2}^2  \nonumber \\
 &\; -   \nu   \frac{T}{N} 
\sum_{l=1}^n  |A\bu^l|_{\LL^2}^2  
+ \sum_{l=2}^n 2
\big( A^{\frac{1}{2}} \Delta_l W \, , \, A^{\frac{1}{2}} \bu^{l-1}\big)  .
\end{align} 
 Let $Y$ be a $V$-valued centered Gaussian random variable with the same distribution as $W(1)$. 
  Then the covariance operator of $Y$ is $Q$ and condition \eqref{K0} is satisfied.                                                                                                                                                                                                                                                                                                                                                                                                                                                                                                                                                                                                                                                                                                                                                          
Using the  scaling and the independence of the time increments $\Delta_l W$, 
we deduce that for any  $\alpha>0$, 
\[ \EE\Big[ \exp\Big( \frac{ \alpha }{\lambda (1-\lambda) }     \sum_{l=1}^N |A^{\frac{1}{2}}  \Delta_lW |_{\LL^2}^2 \Big) \Big] = 
 \Big\{ \EE\Big[ e^{   \frac{ \alpha  T }{\lambda (1-\lambda)\, N } |A^{\frac{1}{2}} Y|_{\LL^2}^2} \Big]\Big\}^{N} .\]
Proposition 2.16 in \cite{daPZab} implies that if $\gamma\in   \big( 0, \frac{1}{2K_0}\big)$ and
 $\tilde{\gamma} \in \big[ \gamma, \frac{1}{2K_0})$, we have 
\[ \EE\big( e^{\gamma |A^{\frac{1}{2}} Y|_{\LL^2}^2}\big) \leq \exp\Big( \frac{1}{2} \sum_{i=1}^\infty \frac{( 2\gamma)^i}{i} K_0^i \Big)  \leq 
\exp\Big[ - \frac{1}{2} \ln\big(1- 2 {\gamma} K_0 \big) \Big]\leq \exp\Big[ \frac{\gamma K_0}{1-2\tilde{\gamma} K_0}\Big] <\infty.\] 
Hence, given $\lambda \in (0,1)$, if    $\frac{ \alpha T}{\lambda (1-\lambda) \, N }  < \frac{1}{2 K_0}$, 
(which is satisfied for any  $\alpha>0$ provided that $N$ is large enough), we obtain
 \begin{equation} 		\label{upper_discret}
\Big\{ \EE \Big( \exp\Big( \frac{ \alpha T }{\lambda (1-\lambda)\, N }   
|A^{\frac{1}{2}} Y|_{\LL^2}^2\Big) \Big\}^N \leq  C
\end{equation}  
for some constant $C$ which does not depend on $N$. Given $\alpha >0$ and $n=2,\cdots, N$, set
\[ M_n:=2\alpha \sum_{l=2}^n  \big( A^{\frac{1}{2}} \Delta_lW, \, A^{\frac{1}{2}} \bu^{l-1}\big) =
 2\alpha \sum_{l=2}^n  \big(  \Delta_lW, \, A \bu^{l-1}\big) . 
\]
Then $(M_n, {\mathcal F}_{t_n}, n=1, ..., N)$ is a discrete martingale. For $s\in [t_l, t_{l+1})$, $l=1, \cdots, N-1$, 
 set $\underline{s}=t_l$ and  $\bu^{\underline{s}}=
\bu^l$.  With these notations, $M_n=\tilde{M}_{t_n}$, where 
\[
 \tilde{M}_t=  2\alpha \int_{t_1}^t \big( A   \bu^{\underline{s}} \ , \,  d  W(s)\big), 
\quad t\in [t_1,T]. 
\]
The process 
$(\tilde{M}_t, {\mathcal F}_t, t\in [t_1,T])$ is a square integrable martingale, such that
\begin{align}  \label{upper_qua-var} 
\langle \tilde{M} \rangle_{t_n} & \leq \;  4\alpha^2 \int_{t_1}^{t_n}  {\rm Tr}\, (Q) \, 
 |A \bu^{\underline{s}}|_{\LL^2}^2   ds 
 =  \; 4\alpha^2 \, {\rm Tr}\,(Q)\,  \frac{T}{N}  \sum_{l=1}^{n-1}   |A\bu^{l}|_{\LL^2}^2 . 
\end{align}
 Using \eqref{max-nabla-ul} we deduce that for $\lambda\in (0,1)$, $\alpha >0$ and $\mu>1$,
\begin{align}		\label{upper_expalpha}
& \exp\Big(  \alpha \max_{1\leq n\leq N} \Big[ |A^{\frac{1}{2}} \bu^n|_{\LL^2}^2 +  \frac{T}{N}  \nu \sum_{l=1}^n 
  |A\bu^l|_{\LL^2}^2 \Big]  \Big) \nonumber \\
 &\quad  \leq \exp\Big( \frac{ \alpha}{1-\lambda} \|u_0\|_V^2\Big) 
\exp\Big( \frac{  \alpha}{\lambda(1-\lambda)}\sum_{l=1}^N |A^{\frac{1}{2}}  \Delta_l W|_{\LL^2}^2 \Big)
\exp\Big( \max_{2\leq n\leq N}  \big[ \mu M_n - \frac{\mu}{2} \langle \tilde{M} \rangle_{t_n}\big] \Big)  \nonumber \\
&\qquad \times  
\exp\Big( \max_{2\leq n\leq N} \Big[  \frac{\mu}{2}  \langle \tilde{M} \rangle_{t_n}
 -  \alpha \frac{T}{N}  \nu \sum_{l=1}^n   |A\bu^l|_{\LL^2}^2 \Big] . 
\end{align} 
{\bf Case 1. $u_0\in V$ is deterministic. } 
Let $u_0\in V$ be deterministic. 
For $\alpha \in \big(0, \frac{\tilde{\alpha}_0}{2} \big)$, 
we may choose $\mu>1$ such that  $\mu \alpha \leq  \frac{\tilde{\alpha}_0}{2} $;  using \eqref{upper_qua-var}
we deduce that for such a choice of $\alpha$ and $\mu$ we have a.s. 
\begin{align*}
\max_{2\leq n\leq N} \Big[  \frac{\mu}{2} \langle \tilde{M}\rangle_{t_n}  & -  \alpha \frac{T}{N}  \nu 
\sum_{l=1}^n  |A\bu^l|_{\LL^2}^2 \Big]
 \leq 
\Big(  2\, \mu \alpha \, {\rm Tr}(Q) \ - \nu \Big)\; \alpha  \frac{T}{N}   \sum_{l=2}^{N-1} 
|A\bu^{l}|_{\LL^2}^2  \leq 0.
\end{align*}
Thus, H\"older's inequality with conjugate exponents $\mu$ and $\frac{\mu}{\mu-1}$ implies  for $\lambda = \frac{1}{2}$
\begin{align}		\label{upper_u0det}
 \EE\Big[ \exp & \Big(  \alpha \max_{0\leq n\leq N}   | A^{\frac{1}{2}}  \bu^l|_{\LL^2}^2 + \frac{T}{N} \nu \sum_{l=1}^n 
   |A\bu^l|_{\LL^2}^2\Big)   \Big]
 \nonumber \\
 \quad & \leq   \; \exp\Big( \alpha|A^{\frac{1}{2}} u_0|_{\LL^2}^2 \Big) 
\Big\{ \EE\Big[ \exp\Big(  \frac{ 4 \mu   \alpha}{\mu-1 }\sum_{l=1}^N |A^{\frac{1}{2}}  \Delta_l W |_{\LL^2}^2 \Big)  \Big] \Big\}^{\frac{\mu-1}{\mu}} \nonumber  \\
&\qquad  \times \Big\{ \EE \Big[ \max_{2\leq n\leq N} \exp\Big(    \mu \tilde{M}_{t_n} - \frac{\mu^2}{2} \langle \tilde{M} \rangle_{t_n} \Big) \Big]\Big\}^{\frac{1}{\mu}}.
\end{align} 
Since $\{ \exp\big(    \mu \tilde{M}_t - \frac{\mu^2}{2} \langle \tilde{M} \rangle_{t} \big)\}_{t\in [t_1,T]}$ is an exponential martingale, choosing
$N$ large enough to ensure   $\frac{ 4\,  \tilde{\alpha}_0\, \mu}{ \mu-1}  \,\frac{T}{ N } < \frac{1}{2 K_0} $, 
\eqref{upper_discret} implies \eqref{exp-moments-Aul_al} for a deterministic $V$-valued initial condition. 
\smallskip

\noindent {\bf Case 2. $u_0$ is random}. 
We next prove \eqref{exp-moments-Aul_al} if $u_0$ is random  and \eqref{gamma0} is satisfied. 

As in Case 1, we impose  $\mu \in \big( 1, \infty)$ and $\alpha >0$  to ensure  $2 \mu \alpha {\rm Tr}\,(Q) -\nu \leq 0$,  so that 
$\frac{\mu}{2}  \langle \tilde{M} \rangle_{t_n}
 -  \alpha \frac{T}{N}  \nu \sum_{l=1}^n   |A\bu^l|_{\LL^2}^2 \leq 0$ a.s.   We then  
 use H\"older's inequality in \eqref{upper_expalpha} with exponents $p_1\in (1,\infty)$, $\mu $ and $p_3\in (1,\infty)$ 
such that $\frac{1}{p_1} + \frac{1}{\mu} + \frac{1}{p_3}=1$, to deduce  
\begin{align}		\label{upper_E_u0alea}
 &\EE\Big[ \exp\Big( \alpha \max_{0\leq l\leq N} | A^{\frac{1}{2}} \bu^l|_{\LL^2}^2  + \alpha \nu \frac{T}{N} \sum_{l=1}^N   |A \bu^l |_{\LL^2}^2 
 \Big) \Big] \leq   \; 
 \Big\{ \EE\Big(  \exp\Big(   \frac{ p_1 \alpha|A^{\frac{1}{2}} u_0|_{\LL^2}^2}{1-\lambda} \Big) \Big\}^{\frac{1}{p_1}} \\ 
 &\quad  \times 
 \Big\{ \EE \Big[ \max_{2\leq n\leq N} \exp\Big(    \mu \tilde{M}_{t_n} - \frac{\mu^2}{2} \langle \tilde{M} \rangle_{t_n} \Big) \Big]\Big\}^{\frac{1}{\mu}} 
\Big\{ \EE\Big[ \exp\Big(  \frac{  p_3 \,  \alpha}{ \lambda(1-\lambda)}\sum_{l=1}^N | A^{\frac{1}{2}} \Delta_l W|_{\LL^2}^2 \Big)  \Big] \Big\}^{\frac{1}{p_3}} .
\nonumber 
\end{align}
To use \eqref{gamma0}, we have to impose  ${\alpha} \, p_1 <\gamma_0$. The constraints  yield 
\[ 1> \frac{1}{p_1} + \frac{1}{\mu} \geq \frac{\alpha}{\gamma_0}  + \frac{2\alpha {\rm Tr}(Q)}{\nu}
 = \alpha \Big(  \frac{1}{\gamma_0} + \frac{2}{\tilde{\alpha}_0} \Big) = \alpha \, \frac{2\gamma_0 + \tilde{\alpha}_0}{\gamma_0
\tilde{\alpha}_0}.\] 
Suppose that $\alpha < \tilde{\beta}_1:= \tilde{\alpha}_0 \, \frac{\gamma_0}{2 \gamma_0+\tilde{\alpha}_0} $ and let $\delta \in (\alpha, \tilde{\beta}_1)$.
Set $p_1:=\frac{\gamma_0}{\delta}$ and  $\mu:= \frac{\tilde{\alpha}_0}{2\alpha}$; then for $1-\lambda := \frac{\alpha}{\delta} \in (0,1)$, 
we have $\frac{p_1 \alpha}{1-\lambda}= p_1\, \delta = \gamma_0$,  $2\mu \alpha {\rm Tr}(Q) -\nu =0$,
while $\frac{1}{p_1} + \frac{1}{\mu} < \delta\, \frac{2\gamma_0 + \tilde{\alpha}_0}{\tilde{\alpha}_0 \gamma_0} < 1$. Finally let $p_3\in (1,+\infty)$ be defined
by $\frac{1}{p_3} = 1- \frac{1}{p_1}-\frac{1}{\mu}$ and $N$ be large enough to ensure $\frac{p_3 \tilde{\beta}_1 T}{\lambda (1-\lambda) N}<\frac{1}{2K_0}$.
 Since $\exp(\mu  M_t - \frac{\mu^2}{2} \langle M\rangle_t)$ is an exponential martingale, we deduce that
all three factors in the right hand side of \eqref{upper_E_u0alea} are finite, which proves \eqref{exp-moments-Aul_al}. 
 \medskip

\noindent {\bf Proof of \eqref{exp-moments-ul_al}} We next indicate the small changes in the above proof that imply \eqref{exp-moments-ul_al}. Let us rewrite \eqref{max-nabla-ul} as follows
\begin{align*} 
 | A^{\frac{1}{2}} \bu^n|_{\LL^2}^2    
 \leq & \frac{1}{1-\lambda} |A^{\frac{1}{2}} u_0|_{\LL^2} ^2  
+ \frac{1}{\lambda (1-\lambda)}  \sum_{l=1}^n |A^{\frac{1}{2}} \Delta_l W\|_{\LL^2}^2 
 -  2 \nu   \frac{T}{N} 
\sum_{l=1}^n  |A\bu^l|_{\LL^2}^2  \\  &
 + \sum_{l=2}^n 2 
\big(  \Delta_l W \, , \, A \bu^{l-1}\big)  .
\end{align*} 
Let $M_n$ be defined as above. 
For $\alpha>0$, $\lambda\in(0,1)$ and $\mu>1$, an argument similar to that proving \eqref{upper_expalpha} yields
\begin{align}\label{EE_max_u0random}
& \exp\Big( \alpha \max_{1\leq n\leq N} |A^{\frac{1}{2}} \bu^n|_{\LL^2} ^2 \Big) 
 \leq \exp\Big( \frac{ \alpha}{1-\lambda} |A^{\frac{1}{2}} u_0|_{\LL^2} ^2\Big) 
\exp\Big( \frac{  \alpha}{\lambda(1-\lambda)}\sum_{l=1}^N |A^{\frac{1}{2}}  \Delta_l W|_{\LL^2}^2 \Big) \\
&\quad \times \exp\Big( \max_{2\leq n\leq N}  \big[ \tilde{M}_{t_n} - \frac{\mu}{2} \langle \tilde{ M} \rangle_{t_n}\big] \Big) 
\; \exp\Big( \max_{2\leq n\leq N} \Big[  \frac{\mu}{2}  \langle \tilde{ M} \rangle_{t_n}
 - 2 \alpha \frac{T}{N}  \nu \sum_{l=1}^n   |A\bu^l|_{\LL^2}^2\Big]  \Big) .  \nonumber 
\end{align}
\noindent {\bf Case 1. $u_0$ is deterministic} The last exponential factor in \eqref{EE_max_u0random} is a.s. upper estimated by 1 provided that 
$\mu \alpha \leq \frac{\nu}{ {\rm Tr}(Q)} =  \tilde{\alpha_0}$. 
If $u_0$ is deterministic, 
the H\"older inequality with conjugate exponents $\mu$ and $\frac{\mu}{\mu-1}$ concludes the proof of \eqref{exp-moments-ul_al} for large $N$.

\noindent {\bf Case 2. $u_0$ is random and \eqref{gamma0} holds. }
Let $\alpha < \tilde{\alpha}_0  \frac{\gamma_0}{ \gamma_0 +  \tilde{\alpha}_0}:= \tilde{\beta}_0$, $\delta \in (\alpha, \tilde{\beta}_0)$ and set $1-\lambda
= \frac{\alpha}{\delta}\in (0,1)$. 
 Set $\mu=\frac{\tilde{\alpha}_0}{\alpha}$ and
$p_1 = \frac{\gamma_0}{\delta} $; then $\frac{1}{p_1} + \frac{1}{\mu} < \frac{\delta}{\gamma_0} + \frac{\delta}{\tilde{\alpha}_0}<1$, 
$\frac{\alpha p_1}{1-\lambda}=\gamma_0$ and  the last exponential factor  in \eqref{EE_max_u0random} is a.s. upper estimated by 1. 
 Define $p_3\in (0,1)$ by $ \frac{1}{p_3}=1- \frac{1}{p_1}-\frac{1}{\mu} $. 
 Using again H\"older's inequality in \eqref{upper_expalpha} with exponents $p_1, \, \mu$ and $p_3$, 
 an argument similar to that used to prove \eqref{exp-moments-Aul_al} for a random initial condition 
concludes the proof of \eqref{exp-moments-ul_al}. The proof of Theorem \ref{exp-mom} is complete. 
 \hfill $\square$
 \medskip

\noindent {\bf Acknowledgements}
 Hakima Bessaih is partially supported by the Simons Foundation grant 582264. \\
 Annie Millet's research has been conducted within the FP2M federation (CNRS FR 2036).


\end{document}